\documentclass[11pt]{article}
\usepackage[a4paper]{geometry}	% Mise en page

\usepackage{amsmath,amssymb,amsthm}
\usepackage{mathtools}
\usepackage{microtype}
\usepackage{hyperref}
\usepackage{dsfont}
\usepackage{comment}
\usepackage{enumitem}
\usepackage{tikz}
\usetikzlibrary{patterns}
\usepackage{authblk} % For author affiliation

\newcommand{\R}{\mathbb{R}}
\newcommand{\N}{\mathbb{N}}

\newcommand{\ra}{\rightarrow}
\newcommand{\wc}{\rightsquigarrow}
\renewcommand{\d}{\,\mathrm{d}}

\newcommand{\lp}{\left(}
\newcommand{\rp}{\right)}
\newcommand{\lc}{\left[}
\newcommand{\rc}{\right]}
\newcommand{\lacc}{\left\{}
\newcommand{\racc}{\right\}}
\newcommand{\labs}{\left|}
\newcommand{\rabs}{\right|}

\DeclareMathOperator\symdif{\triangle}

\DeclareMathOperator{\oh}{o}
\DeclareMathOperator{\Oh}{O}

\DeclareMathOperator{\EE}{\operatorname{\mathsf{E}}}%{\operatorname{\mathbb{E}}}

\newcommand{\PP}{\operatorname{\mathsf{P}}}%{\operatorname{\mathbb{P}}}
\newcommand{\1}{\mathds{1}}

\renewcommand{\mod}{\mathcal{M}}
\newcommand{\param}{\mathcal{R}}

\newcommand{\cA}{\mathcal{A}}
\newcommand{\cF}{\mathcal{F}}
\newcommand{\cS}{\mathcal{S}}

\newcommand{\hPhi}{\widehat{\Phi}}
\newcommand{\hQ}{\widehat{Q}}

\newcommand{\Wnk}{W_{n,k}}

\definecolor{darkteal}{rgb}{0, 0.35, 0.35}

\definecolor{cerulean}{rgb}{0.0, 0.48, 0.65}
\newcommand{\mdf}[1]{\textcolor{black}{#1}}
\definecolor{myblue}{rgb}{0,0,0.5}
\newcommand{\rev}[1]{\textcolor{black}{#1}}
\newcommand{\revtwo}[1]{\textcolor{black}{#1}}

\theoremstyle{plain}
\newtheorem{theorem}{Theorem}
\newtheorem*{theorem*}{Theorem}
\newtheorem{proposition}{Proposition}
\newtheorem{corollary}{Corollary}
\theoremstyle{remark}
\newtheorem{assumption}{Assumption}

\title{An asymptotic expansion \mdf{of} the empirical angular measure \mdf{for bivariate extremal dependence}}
\author[1]{Stéphane Lhaut}
\author[1]{Johan Segers}
\affil[1]{UCLouvain, LIDAM/ISBA, Voie du Roman Pays 20, 1348 Louvain-la-Neuve, Belgium.
E-mails: \href{mailto:stephane.lhaut@uclouvain.be}{stephane.lhaut@uclouvain.be} and \href{mailto:johan.segers@uclouvain.be}{johan.segers@uclouvain.be}}
\date{}

\begin{document}

\maketitle

\begin{abstract}
The angular measure on the unit sphere characterizes the first-order dependence structure of the components of a random vector in extreme regions and is defined in terms of standardized margins. Its statistical recovery is an important step in learning problems involving observations far away from the center. In the common situation that the components of the vector have different distributions, the rank transformation offers a convenient and robust way of standardizing data in order to build an empirical version of the angular measure based on the most extreme observations.
We provide a functional asymptotic expansion for the empirical angular measure in the bivariate case based on the theory of weak convergence in the space of bounded functions. From the expansion, not only can the known asymptotic distribution of the empirical angular measure be recovered, it also enables to find expansions and weak limits for other statistics based on the associated empirical process or its quantile version.
\end{abstract}

%%%%%%%%%%%%%%%%%%%%%%%%%%%%%%%%%%%%%%%%%%%%%%%%%%%%%%%%%%%%%%%%%%%%%%%%
\section{\rev{Introduction and background}}
\label{sec:intro}
%%%%%%%%%%%%%%%%%%%%%%%%%%%%%%%%%%%%%%%%%%%%%%%%%%%%%%%%%%%%%%%%%%%%%%%%

A standard assumption on the joint upper tail of the distribution of a random vector $X = (X_1,X_2)$ in $\R^2$ is that its distribution belongs to the \emph{maximal domain of attraction} of a multivariate extreme value distribution. This hypothesis comprises two parts:
\begin{enumerate}
	\item the marginal distributions of $X$ belong to the maximal domains of attraction of some univariate extreme value distributions;
	\item after marginal transformation of \mdf{$X_1$ and $X_2$} through the probability integral transform or a variation thereof, the joint distribution of the transformed vector belongs to the maximal domain of attraction of a multivariate extreme value distribution with pre-specified margins.
\end{enumerate}
Under the side assumption that the marginal distribution functions $F_1$ and $F_2$ of $X$ are continuous, which we make in this paper, point~2 above only involves the copula of $X$ in view of Sklar's theorem. Point~2 can be imposed on its own, independently of the assumptions on the margins in point~1, and this is what we do in this paper.

\smallskip
\noindent \textbf{Regular variation.} \quad 
\mdf{Our working hypothesis is that the distribution} of the standardized vector $(Z_1,Z_2) := (1/(1-F_1(X_{1})), 1/(1-F_2(X_{2})))$ is \emph{multivariate regularly varying}, i.e., \rev{there exists a Radon measure $\nu$ on $E := [0,+\infty]^2\setminus\{(0,0)\}$ such that
\begin{equation}
\label{eq:multivariateRV}
		s \PP \lc (Z_1,Z_2) \in s \, \cdot \, \rc \xrightarrow{v} \nu(\cdot), \qquad  \text{ when } s \ra \infty, 
\end{equation}
where $\xrightarrow{v}$ denotes \emph{vague convergence} in $E$~\cite[Section~3.4]{resnick1987}, i.e., for every $f: E \ra [0,\infty)$ which is continuous with compact support, we have $s \EE \lc f(s^{-1}(Z_1,Z_2)) \rc \to \int_E f \d\nu$ as $s \to \infty$.} The \emph{exponent measure} $\nu$ appears in the exponent of the cumulative distribution function \mdf{of} the multivariate extreme value distribution to which $(Z_1,Z_2)$ is attracted~\cite[Definition~6.1.7]{dHF2006}.

It can be shown~\cite[Proposition~3.12]{resnick1987} that the vague convergence~\eqref{eq:multivariateRV} is equivalent to the following convergence of probabilities: if $B$ is a relatively compact Borel set of $E$ with $\nu(\partial B) = 0$, then
\begin{equation}
\label{eq:multiavariateRVprobas}
	\lim_{s\ra\infty} s \PP \lc (Z_1,Z_2) \in sB \rc = \nu(B),
\end{equation}
where $sB := \{sx \in E: x \in B\}$.
Intuitively, this relation permits to describe probabilities of events far away from the \mdf{origin} under the law of the standardized random vector $(Z_1,Z_2)$ based on the knowledge of $\nu$. Consequently, statistical inference on the exponent measure $\nu$ is key to model the joint upper tail of $(X_1,X_2)$.

\smallskip
\noindent \textbf{Angular measure.} \quad
It can be shown that $\nu$ is concentrated on $[0,\infty)^2 \setminus \{(0,0)\}$ and satisfies the following homogeneity property~\cite[Thereom~6.1.9]{dHF2006}: for $B$ a Borel set in $E$,
\begin{equation}
\label{eq:homogeneity}
	\nu(cB) = c^{-1} \nu(B), \qquad \text{ for every } c>0,
\end{equation}

Property~\eqref{eq:homogeneity} motivates a transformation to pseudo-polar coordinates. Given any $L^p$-norm, $p \in [1,\infty]$, on $\R^2$, consider $T_p : (z_1,z_2) \in [0,\infty)^2 \setminus \{(0,0)\} \mapsto T_p(z_1,z_2) \in (0,\infty) \times [0,\pi/2]$ defined by $T_p(z_1,z_2) := (\|(z_1,z_2)\|_p, \arctan(z_1/z_2))$. By \eqref{eq:homogeneity}, the push-forward measure of $\nu$ under $T_p$ is a product measure given by
\begin{equation}
\label{eq:polar_nu}
	T_p \# \nu := \nu \circ T_p^{-1} = \nu_{-1} \otimes \Phi_p,
\end{equation}
where $\nu_{-1}$ is the Borel measure on $(0,\infty)$, determined by $\nu_{-1} \lp [r,\infty) \rp = r^{-1}$ for any $r>0$, and $\Phi_p$ is the \emph{angular measure} on $[0, \pi/2]$ defined by
\begin{equation}
\label{eq:defangularmeasure}
	\Phi_p(A) = \nu \big( \{ (z_1,z_2) \in [0,\infty)^2 \setminus \{(0,0)\}: \|(z_1,z_2)\|_p \geq 1, \arctan(\tfrac{z_1}{z_2}) \in A \} \big),
\end{equation}
for every Borel set $A \subseteq [0,\pi/2]$, see ~\cite[Equation~(6.14)]{resnick2006}.
The representation~\eqref{eq:polar_nu} implies that the angular measure $\Phi_p$ fully characterizes $\nu$.
Consequently, we focus on inference for $\Phi_p$.
We refer to~\cite[Chapter~5]{resnick1987} for background on multivariate extremes.

\smallskip
\noindent \textbf{Empirical angular measure.} \quad
\mdf{Given is an independent random sample $(X_{i1},X_{i2})$, for $i \in \{1,\ldots,n\}$, whose common distribution has continuous margins and angular measure $\Phi_p$.}
Our estimate for $\Phi_p$ relies on the probabilistic interpretation of its definition~\eqref{eq:defangularmeasure} obtained by~\eqref{eq:multiavariateRVprobas}, i.e., for any Borel set $A \subseteq [0,\pi/2]$,
\[
	\Phi_p(A) = \lim_{s\ra\infty} s \PP \lc \|(Z_1,Z_2)\|_p \geq s, \arctan(Z_1/Z_2) \in A \rc.
\]
Let $k = k(n) \in \{1,\ldots,n\}$ be such that $k \to \infty$ and $k = \oh(n)$ as $n \to \infty$.
Replace $s$ by $n/k$ and the induced law of $(Z_1,Z_2)$ on $E$ by its empirical counterpart on the empirically standardized sample
\begin{align*}
	(\hat{Z}_{i1},\hat{Z}_{i2}) 
	:= \lp \frac{1}{1-\hat{F}_1(X_{i1})}, \frac{1}{1-\hat{F}_2(X_{i2})} \rp
	= \lp \frac{n}{n+1-R_{i1}}, \frac{n}{n+1-R_{i2}} \rp,
\end{align*}
where $\hat{F}_j(x_j) := n^{-1} \sum_{i=1}^n \1 \{X_{ij} < x_j\}$ and where $R_{ij}$ denotes the rank of $X_{ij}$ in the marginal sample $X_{1j}, \ldots, X_{nj}$.
The \emph{empirical angular measure} of $\Phi_p(\theta) := \Phi_p \lp [0,\theta] \rp$, $\theta \in [0,\pi/2]$ is defined by
\begin{align*}
	\widehat{\Phi}_p(\theta) 
	&:= \frac{n}{k} \frac{1}{n} \sum_{i=1}^n \1 \Big\{ \|(\hat{Z}_{i1},\hat{Z}_{i2})\|_p \geq \tfrac{n}{k}, \arctan(\hat{Z}_{i1}/\hat{Z}_{i2}) \leq \theta \Big\} \\
	&= \frac{1}{k} \sum_{i=1}^n \1 \Big\{ (n+1-R_{i1})^{-p} + (n+1-R_{i2})^{-p} \geq k^{-p}, \frac{n+1-R_{i2}}{n+1-R_{i1}} \leq \tan\theta \Big\}.
\end{align*}
\rev{See~\cite{einmahl2001, einmahl2009maximum} for functional asymptotic considerations on this estimator---part of those results will be recalled below---and~\cite{clemencon2022} for finite sample results.}

\smallskip
\noindent \textbf{Asymptotic distribution theory.} \quad
The asymptotic distribution of the process
\[
	\lacc \sqrt{k}\lp \widehat{\Phi}_p(\theta) - \Phi_p(\theta) \rp: \theta \in [0,\pi/2] \racc
\]
was derived in~\cite{einmahl2001} for the case $p=\infty$ and~\cite{einmahl2009maximum} for a general $p \in [1,\infty]$.
We recall the precise statement below since it forms the backbone for our main result.

The assumptions underlying asymptotic results are best described in terms of a measure $\Lambda$, which can be derived from the exponent measure $\nu$, and defined as follows: for Borel sets $B \subseteq E^{-1} := [0,\infty]^2 \setminus \{(\infty,\infty)\}$,
\[
	\Lambda(B) := \nu \lp \{(z_1,z_2) \in E : (1/z_1,1/z_2) \in B\} \rp.
\]
It corresponds to a uniform standardization of the margins in the multivariate regular variation assumption~\eqref{eq:multivariateRV}, i.e., \rev{if $P$ denotes the law of the random vector $(1-F_1(X_1),1-F_2(X_2))$}, the measure $\Lambda$ arises as the vague limit $s P (s^{-1} \cdot) \xrightarrow{v} \Lambda(\cdot)$ on $E^{-1}$ as $s \ra \infty$.
\rev{The angular measure $\Phi_p$ can be computed in terms of $\Lambda$ using relation
\begin{equation}
\label{eq:Phi_from_Lambda}
	\Phi_p(\theta) = \Lambda(C_{p,\theta}), \qquad \theta \in [0,\pi/2]
\end{equation}
where
\begin{equation}
\label{eq:Cptheta}
	C_{p,\theta} :=
	\begin{cases}
		([0,\infty] \times \{0\}) \cup (\{\infty\} \times [0,1]) 
		&\text{if } \theta = 0, \\
		\lacc (x,y) : 0 \leq x \leq \infty, 0 \leq y \leq \min\{x \tan\theta, y_p(x)\} \racc &\text{if } 0 < \theta < \pi/2, \\
		\lacc (x,y) : 0 \leq x \leq \infty, 0 \leq y \leq y_p(x) \racc &\text{if } \theta = \pi/2,
	\end{cases}
\end{equation}
with
\[
	y_p(x) :=
	\begin{cases}
		\infty &\text{if } x \in [0,1), \\
		\lp 1 + \frac{1}{x^p-1} \rp^{1/p} &\text{if } x \in [1,\infty] \text{ and } p \in [1,\infty), \\
		1 &\text{if } x \in [1,\infty] \text{ and } p = \infty
	\end{cases}
\]
being the smallest $y \geq 1$ solution of the equation $\|(x^{-1},y^{-1})\|_p = 1$.
Note that $x\tan\theta < y_p(x)$ if and only if $x < x_p(\theta) := \|(1,\cot\theta)\|_p$ for any $p \in [1,\infty]$.
We illustrate these quantities in Figure~\ref{fig:C_p} in case $p=1$ and $0 < \theta < \pi/2$. 
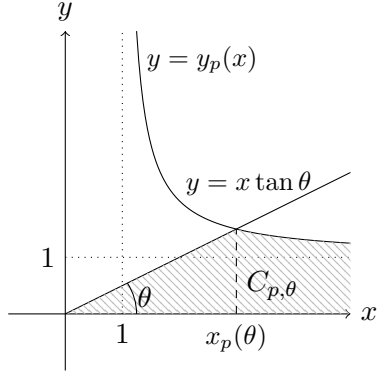
\begin{figure}
\centering
\begin{tikzpicture}[scale=0.75]
\draw[->] (-1,0) -- (5,0);
\draw (5,0) node[right]{$x$};
\draw[->] (0,-1) -- (0,5);
\draw (0,5) node[above]{$y$};
\draw[dotted] (1,0) -- (1,5);
\draw (1,0) node[below]{$1$};
\draw[dotted] (0,1) -- (5,1);
\draw (0,1) node[left]{$1$};
\draw [domain=0:5] plot(\x,{\x/2});
\draw (4.5,2.3) node[left]{\small $y=x\tan\theta$}; %theta vaut plus ou moins 0.5
\draw [domain=1.25:5,samples=100] plot(\x,{1+1/(\x - 1)}); %p vaut 1
\draw (1.25,4.5) node[right]{\small $y=y_p(x)$};
\fill [pattern=north west lines, pattern color=lightgray] (0,0) -- plot [domain=0:3] (\x, \x/2) -- plot [domain=3:5] (\x,{1+1/(\x - 1)}) -- (5,0) -- cycle;
\draw (3,0.5) node[right]{$C_{p,\theta}$};
\draw[dashed] (3,0) -- (3,1.5);
\draw (3,0) node[below]{\small $x_p(\theta)$};
\draw (1.25,0) arc (0:33:1);
\draw (1.7,0.3) node[left]{$\theta$};
\end{tikzpicture}
\caption{Illustration of the set $C_{p,\theta}$ in \eqref{eq:Cptheta} and related quantities for $p=1$ and $0 < \theta < \pi/2$.}
\label{fig:C_p}
\end{figure}
}

\begin{assumption}[Smoothness]
\label{ass:smoothness}
The measure $\Lambda$ is absolutely continuous with respect to the Lebesgue measure with a density $\lambda$ which is continuous on $[0,\infty)^2 \setminus \{(0,0)\}$.
Furthermore,
$
	\Lambda(\{\infty\} \times [0,1]) = \Lambda ([0,1] \times \{\infty\}) = 0
$.
\end{assumption}

Assumption~\ref{ass:smoothness} implies that $\Phi_p$ is concentrated on $(0,\pi/2)$, thereby excluding asymptotic independence. Indeed, a calculation shows that $\Phi_p(\{0\}) = \Lambda(\{\infty\} \times [0,1]) = 0$ and $\Phi_p(\{\pi/2\}) = \Lambda ([0,1] \times \{\infty\}) = 0$.
It also implies that the map $\theta \in (0,\pi/2) \mapsto \Phi_p(\theta)$ has a continuous derivative $\varphi_p$ on $(0,\pi/2)$. In particular, $\lambda$ and $\varphi_p$ are related through the following expression: for every $x,y>0$,
\begin{equation}
\label{eq:densities_relation}
	\lambda(x,y) = \frac{xy}{x^2+y^2} \|(x,y)\|_p^{-1} \varphi_p(\arctan(y/x)).
\end{equation}
\rev{The proof of this identity, which is of independent interest, is given in Appendix~A in the supplement.}
In addition, we see that
\begin{equation}
\label{eq:lambda_homogen}
	\lambda(ax,ay) = a^{-1} \lambda(x,y), \qquad a > 0, \ (x,y) \in [0,\infty)^2 \setminus \{(0,0)\}.
\end{equation}

\begin{assumption}[Vanishing bias]
\label{ass:bias}
If \rev{$c$ denotes the density of $P$}, the quantity
\[
	\mathcal{D}_T(t) := \iint_{\mathcal{L}_T} |tc(tu_1,tu_2) - \lambda(u_1,u_2)| \, \d u_1 \d u_2 \quad 1 \leq T < \infty, \, t>0,
\]
where $\mathcal{L}_T := \{(u_1,u_2) \in [0,T]^2 : u_1 \wedge u_2 \leq 1 \}$, satisfies $\mathcal{D}_{1/t}(t) \ra 0$ as $t \ra 0$ and $\sqrt{k} \mathcal{D}_{n/k}(k/n) \ra 0$ as $n \ra \infty$. \rev{Furthermore, as $n \ra \infty$,}
\[
	\sqrt{k} \, \Phi_p(\tfrac{k}{n}) \ra 0 \qquad \text{and} \qquad
	\sqrt{k} \lp \Phi_p(\tfrac{\pi}{2}) - \Phi_p(\tfrac{\pi}{2}-\tfrac{k}{n}) \rp \ra 0.
\]
\end{assumption}

\rev{Assumption~\ref{ass:bias} will be key to control the difference between $\Phi_p(\theta) = \Lambda(C_{p,\theta})$ and its preasymptotic version $sP(s^{-1} C_{p,\theta})$ for $s>0$ large. In particular, note that the sets of interest $C_{p,\theta}$ are not included in $\mathcal{L}_T$ even for large $T$. However, thanks to some homogeneity property of $\lambda$ in~\eqref{eq:lambda_homogen}, we will be able to cover $C_{p,\theta}$ by a set on which the bias vanishes asymptotically thanks to the first part of the assumption and a tail set whose measure will be small by the second part of the assumption. Details are provided in the proof of our main result.}

\rev{Let $W_\Lambda$ be a centered Wiener process indexed by Borel sets on $E^{-1}$ with ``time'' $\Lambda$ and covariance function $\EE[W_\Lambda(C)W_\Lambda(C')] = \Lambda(C \cap C')$. Let $\ell^\infty(\mathcal{F})$ denote the Banach space of bounded functions from a set $\mathcal{F}$ to $\mathbb{R}$, equipped with the supremum norm. We are now ready to formulate Theorem~3.1 in \cite{einmahl2009maximum}. Here and below, the symbol $\wc$ denotes weak convergence in a metric space as in~\cite{VVV1996}.}
  
\begin{theorem}[Asymptotic distribution of the empirical angular measure]
\label{thm:asymptoticEAM}
Let $p \in [1,\infty]$. 
Under Assumptions~\ref{ass:smoothness} and~\ref{ass:bias}, we  have in $\ell^\infty([0,\pi/2])$ the weak convergence
%and suppose that Assumptions~\ref{ass:smoothness} and~\ref{ass:bias} are satisfied. Then, in $\ell^\infty([0,\pi/2])$, we have the weak convergence
$$
	\lacc \sqrt{k} \lp \widehat{\Phi}_p(\theta) - \Phi_p(\theta) \rp \racc_{\theta \in [0,\pi/2]} \wc \lacc W_\Lambda(C_{p,\theta}) + Z_p(\theta) =: \alpha_p(\theta) \racc_{\theta \in [0,\pi/2]},
$$
where the process $Z_p$ is defined by
\begin{align*}
		&Z_p(\theta) 
		:= \int_0^{x_p(\theta)} \lambda(x,x\tan \theta) \left[W_1(x) \tan \theta - W_2(x \tan \theta)\right] \, \d x \\ 
		&\, +
		\begin{dcases}
			\int_{x_p(\theta)}^\infty \lambda(x,y_p(x))  \left[W_1(x) y_p'(x) - W_2(y_p(x))\right] \, \d x  &\text{if } p < \infty, \\
			\mbox{} - W_1(1) \int_1^{1 \vee \tan \theta} \lambda(1,y) \, \d y - W_2(1) \int_{1 \vee \cot \theta}^\infty \lambda(x,1) \, \d x   &\text{if } p = \infty,
		\end{dcases}
\end{align*}
with $y_p'$ the derivative of $y_p$ and with $W_1(x) := W_\Lambda((0,x] \times (0,\infty])$ and $W_2(y) := W_\Lambda((0,\infty] \times (0,y])$ the marginal processes.
\end{theorem}

\smallskip
\noindent \textbf{Outline.} \quad
\rev{In Section~\ref{sec:main}, our main result is motivated and stated together with some consequences. A general weak convergence result holding in arbitrary dimension which is key to our proof is formulated as well. Section~\ref{sec:proof} contains the proof of the main theorem. Additional proofs are to be found in the supplement.}

%%%%%%%%%%%%%%%%%%%%%%%%%%%%%%%%%%%%%%%%%%%%%%%%%%%%%%%%%%%%%%%%%%%%%%%%
\section{Main results and applications}
\label{sec:main}
%%%%%%%%%%%%%%%%%%%%%%%%%%%%%%%%%%%%%%%%%%%%%%%%%%%%%%%%%%%%%%%%%%%%%%%%

\rev{In many situations, an asymptotic analysis of a statistic based on a given stochastic process requires more than the limiting distribution of this process. It is often more convenient to describe its limiting behavior in terms of an asymptotic expansion.} 

\rev{A typical example comes from the fundamental goodness-of-fit problem. If one seeks to test a parametric model for the angular measure, i.e., testing if
$
	H_0: \Phi_p \in \mod = \{\Phi_{p,r}: r \in \param \subset \R^m \}
$
is a plausible hypothesis, then a standard way to proceed is to use
$
	T_n^\mod :=  \Delta (\hPhi_p, \Phi_{p,{\hat{r}_n}})
$
as a test statistic, where $\Delta$ is a kind of dissimilarity measure and $\hat{r}_n$ is an estimator of the parameter $r \in \param$, and to reject $H_0$ provided $T_n^\mod$ is ``too large''. A way to make that statement precise is to study the asymptotic distribution of $T_n^\mod$ under the null hypothesis and compare the observed statistic to a quantile of the limiting distribution. In many situations (think of Kolmogorov--Smirnov or Anderson--Darling type of statistics), the normalized test statistic $\sqrt{k} T_n^\mod$ will be expressed as a continuous functional of the empirical process $\sqrt{k}(\hPhi_p- \Phi_{p,{\hat{r}_n}})$. Therefore, an analysis to find the asymptotic distribution of $T_n^\mod$ under $H_0$ can be based on determining the asymptotic distribution of the latter process in a sufficiently rich metric space and then applying a continuous mapping theorem. To this end, note that, if $r_0 \in \param$ denotes the true value of the parameter under $H_0$, the process is naturally decomposed as
\[
	\sqrt{k} \, ( \hPhi_p- \Phi_{p,{\hat{r}_n}} ) 
	= \sqrt{k} \, ( \hPhi_p- \Phi_{p,r_0} ) 
	- \sqrt{k} \, ( \Phi_{p,{\hat{r}_n}}-\Phi_{p,r_0} ).
\]
Asymptotic considerations regarding the first part of the decomposition can be obtained on the basis of Theorem~\ref{thm:asymptoticEAM}. However, those need to be obtained \emph{jointly} with the second term to determine the asymptotic distribution of the process of interest. This is only possible if an asymptotic expansion of the first term is available, which has to be combined with assumptions on the estimator $\hat{r}_n$ for the second term.}

\rev{Theorem~\ref{thm:expansionEAM} below provides such an asymptotic expansion. In contrast to the proof of Theorem~\ref{thm:asymptoticEAM} in~\cite{einmahl2001,einmahl2009maximum}, the proof of our result does not require a Skorokhod construction and is shown on the original probability space directly. In particular, it relies on a weak convergence result, holding in arbitrary dimension, which may be of independent interest. In the following proposition, the requirement that $\cA$ is suitably measurable is to be understood as in Theorems~19.14 and~19.28 in \cite{VVV1998} and is explained in more detail in Definition~2.3.3 and Example~2.3.4 in \cite{VVV1996}.}

\begin{proposition}[Weak convergence of the tail empirical process]
\label{prop:general_weak_convergence}
Let $\cA$ be a \rev{suitably measurable Vapnik--Chervonenkis (VC)} class of Borel sets on $[0, \infty)^d$ such that there exists $M > 0$ with
\[
	A \subseteq L_M := \{x \in [0, \infty)^d: \min(x_1,\ldots,x_d) \leq M \}, \qquad A \in \cA.
\]
Let $P$ be a Borel probability measure on $[0,1]^d$ with uniform margins such that there exists a Radon measure $\Lambda$ on $[0,\infty]^d \setminus \{(\infty, \ldots, \infty)\}$ satisfying
\[
	\lim_{t \ra 0} \labs t^{-1} P(tB) - \Lambda(B) \rabs = 0,
\]
for any Borel set bounded away from infinity with $\Lambda(\partial B) = 0$, and
\begin{equation}
\label{eq:uniform_rv}
	\lim_{t \ra 0} \sup_{A, A' \in \cA} \labs t^{-1} P(t(A \symdif A')) - \Lambda(A \symdif A') \rabs = 0,
\end{equation}
Let $P_n$ be the empirical measure of an i.i.d. random sample of size $n \in \N$ from $P$. Let $k=k(n) \in \N$ be such that $k \ra \infty$ and $k = \oh(n)$ as $n \ra \infty$.
Then, in $\ell^\infty(\cA)$,
\begin{equation}
\label{eq:weak_convergence_underlying}
	\lp 
	\Wnk(A) :=
		\sqrt{k} \lacc \tfrac{n}{k} P_n \lp \tfrac{k}{n} A \rp
		- \tfrac{n}{k} P \lp \tfrac{k}{n} A \rp \racc 
	\rp_{A \in \cA}
	\wc \lp W_\Lambda(A) \rp_{A \in \cA},
\end{equation}
where $W_\Lambda$ is tight, centred Gaussian process on $\cA$ with covariance function $\EE[W_\Lambda(A) \, W_\Lambda(A')] = \Lambda(A \cap A')$.
Moreover, the process $\Wnk$ is asymptotically equicontinuous in probability with respect to the semi-metric $\rho(A,A') := \Lambda(A \symdif A')$, i.e., for all $\epsilon > 0$,
\begin{equation}
\label{eq:asy_equi}
	\lim_{\delta \downarrow 0} \limsup_{n \ra \infty} \PP^* \lc \sup_{A, A' \in \cA, \rho(A,A') < \delta} \labs \Wnk(A) - \Wnk(A') \rabs > \epsilon \rc = 0.
\end{equation}
\end{proposition}

\rev{The proof of Proposition~\ref{prop:general_weak_convergence} is deferred to Appendix~B in the supplement. It relies essentially on an application of Theorem~19.28 in~\cite{VVV1998}, which is also recalled.}

\rev{Before stating our main result, we introduce and recall some notation. Let $U := (1-F_1(X_1),1-F_2(X_2))$ be the uniform version of $X$ and let $P$ denote its law. Let $P_n$ denote the empirical measure of an independent sample $U_1,\ldots,U_n$ of $P$. For $j \in \{1,2\}$ and $u \in [0,1]$, let the associated marginal empirical cumulative distribution functions be denoted by 
$
	\Gamma_{jn}(u) := \frac{1}{n} \sum_{i=1}^n \1\{U_{ij} \leq u\}.
$
We set $\Gamma_{jn}(u) := u$ for $u > 1$. 
Let $Q_{jn}$ denote the left-continuous inverse of $\Gamma_{jn}$, that is, for $j \in \{1,2\}$ and $y\geq 0$,
$
	Q_{jn}(y) := \inf \lacc u \in [0,\infty): \Gamma_{jn}(u) \geq y \racc.
$
In particular, for $y > 1$, we have $Q_{jn}(y) = y$. We set $Q_{jn}(y) := 0$ for $0 \leq y \leq (2n)^{-1}$, by convention.
We define the tail marginal empirical processes
$w_{jn}(x) := \sqrt{k} (\tfrac{n}{k} \Gamma_{jn} (\tfrac{k}{n} x) - x)$ and $v_{jn}(x) := \sqrt{k} (\tfrac{n}{k} Q_{jn} (\tfrac{k}{n} x) - x)$ for $j \in \{1,2\}$ and $x\geq 0$.}

\rev{The asymptotic behavior of those processes can be understood using Proposition~\ref{prop:general_weak_convergence}. For example, it is nearly immediate that under Assumptions~\ref{ass:smoothness} and~\ref{ass:bias}, the weak convergence relations
\begin{equation}
\label{eq:marginal_weak_convergences}
	\lacc w_{jn}(x) : x \in [0,M] \racc \wc \lacc W_j(x) : x \in [0,M] \racc, \qquad j \in \{1,2\},
\end{equation}
hold in $\ell^\infty([0,M])$ for $M>1$. By the functional delta method, we also have
\begin{equation}
\label{eq:quantile_weak_convergences}
	\sup_{x \in [0,M]} \labs v_{jn}(x) + w_{jn}(x) \rabs = \oh_{\PP}(1), \qquad j \in \{1,2\},
\end{equation}
meaning that the limiting behavior of the quantile processes is also a consequence of our result. Useful is to have asymptotic results also for the weighted versions of these processes. In particular, we will need the following weak convergence result whose proof can be found in Appendix~C in the supplement. It also relies on an application of Theorem~19.28 in~\cite{VVV1998}.
\begin{proposition}[Weighted weak convergence of tail marginal empirical processes]
\label{prop:weighted}
Let $U_1,\ldots,U_n$ denote uniformly distributed random variables on $[0,1]$. Let
\[
	\Gamma_n(u) :=
	\begin{cases}
		\frac{1}{n} \sum_{i=1}^n \1\{U_i \leq u\} \qquad &\text{if } u \in [0,1] \\
		u \qquad &\text{if } u > 1
	\end{cases}
\]
denote the associated (modified) empirical cumulative distribution function and let
\[
	Q_n(u) := 
	\begin{cases}
		\inf \lacc x \geq 0: \Gamma_n(x) \geq u \racc
		&\text{if $u > (2n)^{-1}$,} \\
		0 & \text{if $0 \le u \le (2n)^{-1}$.}
	\end{cases}
\]
be the (modified) left-continuous inverse of $\Gamma_n$. For $x \geq 0$, consider the processes
\[
	w_n(x) := \sqrt{k} \lp \tfrac{n}{k} \Gamma_n ( \tfrac{k}{n} x ) - x \rp \qquad \text{and} \qquad
	v_n(x) := \sqrt{k} \lp \tfrac{n}{k} Q_n ( \tfrac{k}{n} x ) - x \rp.
\]
Let $w: x \in (0,\infty) \mapsto w(x) := 1/(x^\delta \vee x^\eta)$ with $0 < \delta < 1/2 < \eta$ and set $w(0):= 0$. 
In $\ell^\infty([0,\infty))$, we have
\begin{equation} 
	\label{eq:w_jn_poids}
	\lacc w_n(x)w(x): x \in [0,\infty) \racc \wc \lacc W_j(x)w(x): x \in [0,\infty) \racc,
\end{equation}
as $n \to \infty$. Furthermore, for any $M > 1$,
\begin{equation}
	\label{eq:quantile_weak_convergences_weighted}
	\sup_{x \in [0,M]} \frac{\labs v_{n}(x) + w_{n}(x) \rabs}{x^\delta} = \oh_{\PP}(1).
\end{equation}
\end{proposition}
}

\begin{theorem}[Asymptotic expansion of the empirical angular measure]
\label{thm:expansionEAM}
Under Assumptions~\ref{ass:smoothness} and~\ref{ass:bias}, we have for any $p \in [1,\infty]$
\[
	\sup_{\theta \in [0,\pi/2]} \labs \sqrt{k}\lp \hPhi_p(\theta) - \Phi_p(\theta) \rp - E_{n,p}(\theta) \rabs = \oh_{\PP}(1),
\]
as $n \to \infty$, where
\begin{align*}
	\lefteqn{
	E_{n,p}(\theta) := \sqrt{k} \lp \tfrac{n}{k} P_n \lp \tfrac{k}{n} C_{p,\theta} \rp - \tfrac{n}{k} P\lp \tfrac{k}{n} C_{p,\theta} \rp \rp 
	} \\
	&\quad + \int_0^{x_p(\theta)} \lambda(x,x\tan\theta) \lacc w_{1n}(x) \tan\theta - w_{2n} \lp x \tan\theta \rp \racc \d x \\
	&\quad + 
	\begin{dcases}
		\int_{x_p(\theta)}^\infty \lambda(x,y_p(x)) \lacc y_p'(x) w_{1n}(x) - w_{2n} \lp y_p(x) \rp \racc \d x & \text{ if } p < \infty, \\
		\mbox{} - w_{1n}(1) \int_1^{1 \vee \tan\theta} \lambda(1,y) \d y - w_{2n}(1) \int_{1 \vee \cot\theta}^\infty \lambda(x,1) \d x & \text{ if } p = \infty.
	\end{dcases}
\end{align*}
\end{theorem}

Note the similarity with Theorem~\ref{thm:asymptoticEAM}. In fact, Theorem~\ref{thm:asymptoticEAM} is a consequence of Theorem~\ref{thm:expansionEAM}.
In the space $\ell^\infty([0,\pi/2]) \times \ell^\infty([0,\infty))^2$, the joint weak convergence
\begin{multline*}
	\lp \sqrt{k} \{ \tfrac{n}{k} P_n (\tfrac{k}{n} C_{p,\cdot}) - \tfrac{n}{k} P(\tfrac{k}{n} C_{p,\cdot}) \}, w_{1n}(\cdot)w(\cdot), w_{2n}(\cdot)w(\cdot) \rp \\
	\wc \lp W_\Lambda(C_{p,\cdot}), W_1(\cdot)w(\cdot), W_2(\cdot)w(\cdot) \rp,
\end{multline*}
where $w$ denotes the same weight function as in Proposition~\ref{prop:weighted}, can be obtained from Propositions~\ref{prop:general_weak_convergence} and~\ref{prop:weighted} as explained in the proof of Theorem~\ref{thm:expansionEAM}. Writing $w_{jn} = w_{jn} w \frac{1}{w}$ and noting that the integrals $\int_0^{x_p(\theta)} \lambda(x, x \tan \theta) \, \frac{1}{w(x)} \d x$ and so on are finite, the continuous mapping theorem then yields $E_{n,p} \wc \alpha_p$ in $\ell^\infty([0,\pi/2])$, as stated in Theorem~\ref{thm:asymptoticEAM}.

\medskip
\noindent \textbf{The quantile process of the angular probability measure.} \quad
\rev{In the goodness-of-fit problem mentioned in the beginning of this section, many statistics do not involve the original process directly, but rather its quantile version. For example, one may think of a procedure based on Wasserstein distances~\cite{bobkov2019one} which admits an explicit representation in terms of quantiles for measures supported on a one-dimensional space. To find the limit distribution of such a test statistic, one needs an expansion of the quantile process. In our case, since the total mass of $\Phi_p$ is unknown, it is convenient to consider the \emph{angular probability measure} and its empirical counterpart, with cumulative distributions functions $Q_p(\theta) := \Phi_p(\theta) / \Phi_p(\pi/2)$ and $\hQ_p(\theta) := \widehat{\Phi}_p(\theta) / \widehat{\Phi}_p(\pi/2)$, respectively, for $\theta \in [0,\pi/2]$. The asymptotic expansion of $\hQ_p$ follows from the one of $\widehat{\Phi}_p$ in Theorem~\ref{thm:expansionEAM} and Slutsky's lemma. The proof of the Corollary~\ref{cor:eapm} is standard and is omitted for brevity. In Corollary~\ref{cor:quantile}, we analyze the associated quantile function estimate $\hQ_p^{-1}$; a detailed proof is to be found in Appendix~D in the supplement.}

\begin{corollary}[Asymptotic expansion of empirical angular probability measure]
\label{cor:eapm}
Under the same assumptions as Theorem~\ref{thm:expansionEAM}, we have, as $n \ra \infty$,
\[
	\sup_{\theta \in [0,\pi/2]} \labs \sqrt{k} \lp \hQ_p(\theta)) - Q_p(\theta) \rp - \frac{E_{n,p}(\theta) \Phi_p(\tfrac{\pi}{2}) - \Phi_p(\theta) E_{n,p}(\tfrac{\pi}{2})}{\Phi_p(\tfrac{\pi}{2})^2} \rabs = \oh_{\PP}(1).
\]
\end{corollary}
\begin{corollary}[Asymptotic expansion of empirical angular quantile process]
\label{cor:quantile}
Under the same assumptions as Theorem~\ref{thm:expansionEAM} and the additional requirement that $\varphi_p > 0$ on $(0,\pi/2)$, we have for any $[v,w] \subset (0,1)$, as $n \ra \infty$,
%\begin{equation*}
\begin{multline*}
	\sup_{u \in [v,w]} \Biggl| \sqrt{k} \lp \hQ_p^{-1}(u) - Q_p^{-1}(u) \rp
	- \frac{\Phi_p( Q_p^{-1}(u)) \, E_{n,p}(\tfrac{\pi}{2}) - E_{n,p}( Q_p^{-1}(u) ) \, \Phi_p(\tfrac{\pi}{2})}{\varphi_p(Q_p^{-1}(u)) \, \Phi_p(\tfrac{\pi}{2})} \Biggr| 
	\\
	= \oh_{\PP}(1).
\end{multline*}
%\end{equation*}
\end{corollary}
%\rev{A detailed proof of Corollary~\ref{cor:quantile} is to be found in Appendix~D in the supplement.}

%%%%%%%%%%%%%%%%%%%%%%%%%%%%%%%%%%%%%%%%%%%%%%%%%%%%%%%%%%%%%%%%%%%%%%%%
\section{Proof of Theorem~\ref{thm:expansionEAM}}
\label{sec:proof}
%%%%%%%%%%%%%%%%%%%%%%%%%%%%%%%%%%%%%%%%%%%%%%%%%%%%%%%%%%%%%%%%%%%%%%%%

The key idea behind the proof is the following decomposition that can be found in~\cite[Equation~12]{einmahl2001} or~\cite[Equation~3.4]{einmahl2009maximum}: for each $\theta \in [0,\pi/2]$,
\begin{align}
\nonumber
	\sqrt{k} \lp \hPhi_p(\theta) - \Phi_p(\theta) \rp
	&= \sqrt{k} \lp \tfrac{n}{k} P_n \lp \tfrac{k}{n} \hat{C}_{p,\theta} \rp - \tfrac{n}{k} P\lp \tfrac{k}{n} \hat{C}_{p,\theta} \rp \rp \\
\nonumber
	&\qquad + \sqrt{k} \lp \tfrac{n}{k} P \lp \tfrac{k}{n} \hat{C}_{p,\theta} \rp - \Lambda \lp \hat{C}_{p,\theta} \rp \rp \\
\nonumber
	&\qquad + \sqrt{k} \lp \Lambda \lp \hat{C}_{p,\theta} \rp - \Lambda \lp C_{p,\theta} \rp \rp \\
\label{eq:threeterms}
	&=: V_{n,p}(\theta) + r_{n,p}(\theta) + Z_{n,p}(\theta),
\end{align}
where the set $C_{p,\theta}$ is estimated by
\begin{align}
	\hat{C}_{p,\theta} := \Bigl\{ (x,y) : 0 \leq x \leq \infty, 0 \leq &y \leq \tfrac{n}{k} Q_{2n} \lp \Gamma_{1n} \lp \tfrac{k}{n}x \rp \tan\theta \rp, \nonumber \\
	&y \leq \tfrac{n}{k} Q_{2n} \lp \tfrac{k}{n} y_p \lp \tfrac{n}{k} \Gamma_{1n} \lp \tfrac{k}{n}x \rp \rp \rp \Bigr\} \label{eq:def_hat_sets}.
\end{align}
The three terms in the decomposition~\eqref{eq:threeterms} will be called the \emph{stochastic term}, \emph{bias term}, and \emph{random set term}, respectively.

\smallskip
\noindent \rev{\textbf{Outline of the proof.} \quad
In Section~\ref{sec:weak}, we state the main underlying weak convergence result needed for the proof, based on Proposition~\ref{prop:general_weak_convergence}. Unlike~\cite{einmahl2001,einmahl2009maximum}, we prove our asymptotic expansion without requiring a Skorokhod construction but work with the original sample space and random variables directly. Technically, it implies that we have to work with weak convergence instead of the almost sure representations arising from the Skorokhod construction, which makes some continuity arguments a bit more delicate, but the proof becomes more direct.
In Section~\ref{sec:eachterm}, we treat each term of~\eqref{eq:threeterms} separately. 
In Appendix E of the supplement, it is argued that we can reduce the analysis to the case $\theta \in [0, \pi/4]$, by symmetry of the problem. Hence, in the following, we only focus on that region.}

\subsection{Weak convergence of a tail empirical process}
\label{sec:weak}

\rev{As in \cite{einmahl2009maximum}, the main idea of the proof is to approximate the random sets $\hat{C}_{p,\theta}$ by those in a carefully constructed VC class of sets built from a covering of the positive real line by (small) intervals on which the random perturbations of the original sets $C_{p,\theta}$ are manageable.
Fix any $\Delta \in \{1, 1/2, 1/3, \ldots \}$ and $M > 1$ and let $\cA = \cA(\Delta,M)$ be the following VC class of sets, already introduced in~\cite{einmahl2009maximum}. For $m \in \{0, 1, 2, \ldots, 1/\Delta - 1\}$, define the intervals
\begin{equation}
\label{eq:I_intervals}
	I_\Delta(m,\theta) :=
	\begin{cases}
	[m\Delta x_p(\theta), (m+1)\Delta x_p(\theta)] \qquad &\text{if } \theta \in (0,\pi/4], \\
	[0,\infty) \qquad &\text{if } \theta = 0 \text{ and } m = 0, \\
	\emptyset \qquad &\text{if } \theta = 0 \text{ and } m > 0.
	\end{cases}
\end{equation}
and
\begin{equation}
\label{eq:J_intervals}
	J_\Delta(m) := [y_p(1 + (2^{1/p} - 1)(m+1) \Delta), y_p(1 + (2^{1/p} - 1)m \Delta)].
\end{equation}
Set $\tilde{\cA}$ to be the class of sets containing the following sets:
\begin{itemize}
	\item $\bigcup_{m=0}^{1/\Delta-1} \lacc (x,y): x \in I_\Delta(m,\theta), 0\leq y \leq x\tan\theta + B_m(x\tan\theta)^{1/16} \racc$ for some $\theta \in [0,\pi/4]$ and $B_0,\ldots,B_{1/\Delta-1} \in [-1,1]$.
	\item $\bigcup_{m=0}^{1/\Delta-1} \lacc (x,y): x \in J_\Delta(m), x \geq x_p(\theta), 0\leq y \leq y_p(x)(1 + K_m) \racc$ for some $\theta \in [0,\pi/4]$ and $K_0,\ldots,K_{1/\Delta-1} \in [-1/2,1/2]$.
	\item $\{(x,y): x \leq a\}$, $\{(x,y): y \leq a\}$ and $\{(x,y): x \leq a \text{ or }y \leq a\}$ for some $a \in [0,M]$.
\end{itemize}
Next, define $\tilde{\cA}_s := \{A_s: A \in \tilde{\cA}\}$ where for $A \in \tilde{\cA}$, $A_s := \{(x,y): (y,x) \in A\}$. Finally, let $\cA = \tilde{\cA} \cup \tilde{\cA}_s$.}

\rev{As required in Proposition~\ref{prop:general_weak_convergence}, every set in $\cA$ is contained in $L_M = \{x \in [0, \infty)^d: \min(x_1,\ldots,x_d) \leq M \}$, provided $M>1$ is picked large enough---which is not restrictive since we will consider it large later. Furthermore, the other conditions underlying Proposition~\ref{prop:general_weak_convergence} are easily shown to hold in view of the assumptions. In particular, the uniform regular variation condition~\eqref{eq:uniform_rv} is satisfied thanks to the convergence in total variation in Assumption~\ref{ass:bias}. Consequently, we can conclude that \eqref{eq:weak_convergence_underlying} and~\eqref{eq:asy_equi} hold in this setting.}

\subsection{Asymptotic expansions of the three terms in \eqref{eq:threeterms}}
\label{sec:eachterm}

\noindent \textbf{The stochastic term $V_{n,p}(\theta)$.} \quad
Recall $\Wnk$ in Proposition~\ref{prop:general_weak_convergence} and note that $V_{n,p}(\theta) = \Wnk(\hat{C}_{p,\theta})$.
Our aim is to show that, as $n \ra \infty$,
\begin{equation}
\label{eq:expansion_stochastic_term}
	\sup_{\theta \in [0,\pi/4]} \labs \Wnk(\hat{C}_{p,\theta}) - \Wnk(C_{p,\theta}) \rabs = \oh_{\PP}(1).
\end{equation}
We will consider separately $V_{n,p,1}$ and $V_{n,p,2}$, where $V_{n,p,i}(\theta) := \Wnk(\hat{C}_{p,\theta,i})$ with $\hat{C}_{p,\theta,1} :=  \hat{C}_{p,\theta} \cap \{ (x,y) : x < x_p(\theta) \}$ and $\hat{C}_{p,\theta,2} :=  \hat{C}_{p,\theta} \cap \{ (x,y) : x \geq x_p(\theta) \}$. In particular, it would follow from the triangle inequality that the expansion~\eqref{eq:expansion_stochastic_term} is satisfied, provided that
\begin{equation}
\label{eq:expansion_stochastic_term_separately}
	\sup_{\theta \in [0,\pi/4]} \labs V_{n,p,i}(\theta) - \Wnk(C_{p,\theta,i}) \rabs = \oh_{\PP}(1), \qquad i \in \{1,2\},
\end{equation}
where $C_{p,\theta,1} := C_{p,\theta} \cap \{ (x,y) : x < x_p(\theta) \}$ and $C_{p,\theta,2} := C_{p,\theta} \cap \{ (x,y) : x \geq x_p(\theta) \}$.

Let us first consider $i=1$. Recall equation~\eqref{eq:def_hat_sets} and observe that, \mdf{since $\tfrac{n}{k} \Gamma_{jn} (\tfrac{k}{n} x) = x + w_{jn}(x)/\sqrt{k}$ and $\tfrac{n}{k} Q_{jn} (\tfrac{k}{n} x) = x + v_{jn}(x)/\sqrt{k}$,}
\begin{align*}
	\tfrac{n}{k} Q_{2n} \lp \Gamma_{1n} \lp \tfrac{k}{n}x \rp \tan\theta \rp &= x\tan\theta + \frac{z_{n,\theta}(x)}{\sqrt{k}}, \\
	\tfrac{n}{k} Q_{2n} \lp \tfrac{k}{n} y_p \lp \tfrac{n}{k} \Gamma_{1n} \lp \tfrac{k}{n}x \rp \rp \rp &= y_p(x) + \frac{s_{n,p}(x)}{\sqrt{k}},
\end{align*}
where
\begin{align}
	z_{n,\theta}(x) &:= w_{1n}(x) \tan\theta + v_{2n} \lp x\tan\theta + \frac{w_{1n}(x) \tan\theta}{\sqrt{k}} \rp \label{eq:def_z_n}, \\
	s_{n,p}(x) &:= \sqrt{k} \lacc y_p \lp x + \frac{w_{1n}(x)}{\sqrt{k}} \rp - y_p(x) \racc + v_{2n} \lp y_p \lp x + \frac{w_{1n}(x)}{\sqrt{k}} \rp \rp. \label{eq:def_s_n}
\end{align}
Now set, for any $m \in \{0, 1, 2, \ldots, 1/\Delta - 1\}$, with the convention that $0/0 = 0$,
\begin{align*}
	V_{m,\Delta,\theta}^+ &:= \sup_{x \in I_\Delta(m,\theta)} \frac{z_{n,\theta}(x) \wedge \lp s_{n,p}(x) + \sqrt{k} (y_p(x) - x\tan\theta) \rp}{(x\tan\theta)^{1/16}}, \\
	 V_{m,\Delta,\theta}^- &:= \inf_{x \in I_\Delta(m,\theta)} \frac{z_{n,\theta}(x) \wedge \lp s_{n,p}(x) + \sqrt{k} (y_p(x) - x\tan\theta) \rp}{(x\tan\theta)^{1/16}},
\end{align*}
\rev{where the intervals $I_\Delta(m,\theta)$ were defined in~\eqref{eq:I_intervals}.} For fixed $\theta$, note that these intervals cover $[0,x_p(\theta)]$, on which $y_p(x) \geq x\tan\theta$. It follows that on an event whose probability tends to one, $V_{m,\Delta,\theta}^\pm$ is equal to the same expression as in its definition but with the numerator simplified to $z_{n,\theta}(x)$. In the following, we work with these simplified definitions.
Moreover, the weak convergence
\begin{multline} 
\label{eq:z_n_poids}
	\lacc \frac{|z_{n,\theta}(x)|}{(x\tan\theta)^{1/16}} : \theta \in [0,\pi/4], x \in [0,x_p(\theta)] \racc \\
	\wc
	\lacc \frac{W_1(x) \tan\theta - W_2(x\tan\theta)}{(x\tan\theta)^{1/16}} : \theta \in [0,\pi/4], x \in [0,x_p(\theta)] \racc
\end{multline}
holds in the Banach space of bounded functions on the indicated domain.
To see~\eqref{eq:z_n_poids}, recall the definition of $z_{n,\theta}$ in ~\eqref{eq:def_z_n} and note that for any $\theta \in [0,\pi/4]$,
\begin{multline*}
	\frac{z_{n,\theta}(x)}{(x\tan\theta)^{1/16}} = \frac{w_{1n}(x)}{x^{1/16}}(\tan\theta)^{15/16} \\
	+ \frac{v_{2n} \lp x\tan\theta + \frac{w_{1n}(x) \tan\theta}{\sqrt{k}} \rp}{ \lp x\tan\theta + \frac{w_{1n}(x) \tan\theta}{\sqrt{k}} \rp^{1/4}} \lp (x\tan\theta)^{3/4} + \frac{(\tan\theta)^{3/4}}{\sqrt{k}} \frac{w_{1n}(x)}{x^{1/4}} \rp^{1/4}.
\end{multline*}
Using equations~\eqref{eq:w_jn_poids} and~\eqref{eq:quantile_weak_convergences_weighted} in Proposition~\ref{prop:weighted} together with the fact that $x\tan\theta \leq 2^{1/p}$ on the said domain yields~\eqref{eq:z_n_poids}.
In particular, we have
\begin{equation}
\label{eq:z_n_boundedness}
	\sup_{\theta \in [0,\pi/4]} \sup_{x \in [0,x_p(\theta)]} \frac{|z_{n,\theta}(x)|}{(x\tan\theta)^{1/16}} = \Oh_{\PP}(1),
	\qquad n \to \infty.
\end{equation}
Defining
\[
	M_{\Delta,\theta}^\pm := \bigcup_{m=0}^{\Delta - 1} \lacc (x,y) : x \in I_\Delta(m,\theta), 0 \leq y \leq x\tan\theta + \frac{V_{m,\Delta,\theta}^\pm}{\sqrt{k}} (x\tan\theta)^{1/16} \racc,
\]
it is easy to see that for any $\theta \in [0,\pi/4]$ \mdf{we have $M_{\Delta,\theta}^- \subseteq \hat{C}_{\theta,p,1} \subseteq M_{\Delta,\theta}^+$ and thus}
\begin{equation}
\label{eq:bounding_stochastic_1}
	V_{n,\Delta,1}^-(\theta) - R_{n,\Delta,1}(\theta) \leq V_{n,p,1}(\theta) \leq V_{n,\Delta,1}^+(\theta) + R_{n,\Delta,1}(\theta),	
\end{equation}
where
\begin{align}
	V_{n,\Delta,1}^\pm(\theta) &:= \rev{\Wnk(M_{\Delta,\theta}^\pm)}	\label{eq:bounding_processes_1}, \\
	R_{n,\Delta,1}(\theta) &:= \sqrt{k} \tfrac{n}{k} P \lp \tfrac{k}{n} (M_{\Delta,\theta}^+ \setminus M_{\Delta,\theta}^-) \rp. \label{eq:rest_1}
\end{align}

Observe that, due to Assumption~\ref{ass:bias} and the behavior of the marginal tail empirical and quantile processes,
\[
	\sup_{\theta \in [0,\pi/4]} \labs R_{n,\Delta,1}(\theta) - \sqrt{k} \Lambda(M_{\Delta,\theta}^+ \setminus M_{\Delta,\theta}^-) \rabs = \oh_{\PP}(1).
\]
Similar computations as in~\cite[Section~3.1.1]{einmahl2001} show that
\[
	\sqrt{k} \Lambda(M_{\Delta,\theta}^+ \setminus M_{\Delta,\theta}^-) \leq 16 \times 2^{1/p} \sup_{y \geq 0} \lambda(1,y) \max_{m = 0,\ldots,\Delta - 1} (V_{m,\Delta,\theta}^+ - V_{m,\Delta,\theta}^-),
\]
where $\sup_{y \geq 0} \lambda(1,y) < \infty$ since $\lambda$ is continuous on $[0,\infty)^2 \setminus \{(0,0)\}$ thanks to Assumption~\ref{ass:smoothness} and $\lim_{y \ra \infty} \lambda(1,y) = 0$ (the latter limit following from the homogeneity of $\lambda$ in~\eqref{eq:lambda_homogen}).
Since $V_{m,\Delta,0}^+ = V_{m,\Delta,0}^- = 0$ by convention for any value of $m$ and $\Delta$, we deduce from~\eqref{eq:z_n_poids} that, for every $\epsilon > 0$,
\[
	\lim_{\Delta \ra 0} \limsup_{n \ra \infty} \PP^* \lc \sup_{\theta \in [0,\pi/4]} \max_{m = 0,\ldots,\Delta - 1} (V_{m,\Delta,\theta}^+ - V_{m,\Delta,\theta}^-) > \epsilon \rc = 0.
\]
Collecting the latter three equations, we conclude that, for every $\epsilon>0$,
\begin{equation} \label{eq:banane}
	\lim_{\Delta \ra 0} \limsup_{n \ra \infty} \PP^* \lc \sup_{\theta \in [0, \pi/4]} R_{n,\Delta,1}(\theta) > \epsilon \rc = 0.
\end{equation}

Furthermore, since~\eqref{eq:z_n_boundedness} implies that on an event with probability tending to one we have $M_{\Delta,\theta}^\pm \in \cA$ for all $\theta \in [0,\pi/4]$, it follows that for all $\epsilon, \delta > 0$, we have
\begin{multline} \label{eq:limsup3}
	\limsup_{n \ra \infty} \PP^* \lc \max_{\sigma \in \{-,+\}} \sup_{\theta \in [0, \pi/4]} \labs \Wnk(M_{\Delta,\theta}^\sigma) - \Wnk(C_{p,\theta,1}) \rabs > \epsilon \rc \\
	\leq \limsup_{n \ra \infty} \PP^* \lc  \max_{\sigma \in \{-,+\}} \sup_{\theta \in [0, \pi/4]} \rho(M_{\Delta,\theta}^\sigma,C_{p,\theta,1}) \geq \delta \rc \\
	+ \limsup_{n \ra \infty} \PP^* \lc \sup_{A, A' \in \cA, \rho(A,A') < \delta} \labs \Wnk(A) - \Wnk(A') \rabs > \epsilon \rc.
\end{multline}
Standard computations lead to
\begin{multline*}
	\sup_{\theta \in [0, \pi/4]} \rho(M_{\Delta,\theta}^\pm, C_{p,\theta,1}) \\ \leq \frac{16 \times 2^{1/p}}{\sqrt{k}} \sup_{y \geq 0} \lambda(1,y) \sup_{\theta \in [0, \pi/4]} \max_{m = 0,\ldots,\Delta - 1} |V_{m,\Delta,\theta}^\pm| = \oh_{\PP}(1),
\end{multline*}
so that the first limsup on the right hand side of~\eqref{eq:limsup3} is zero. The second limsup can be made arbitrary close to zero as $\delta \downarrow 0$ by~\eqref{eq:asy_equi}.

Combining everything leads to~\eqref{eq:expansion_stochastic_term_separately} for $i=1$. Indeed, for $\epsilon > 0$, we have
\begin{multline*}
	\PP^* \lc \sup_{\theta \in [0, \pi/4]} \labs V_{n,p,1}(\theta) - W_{n,k}(C_{p,\theta,1}) \rabs > \epsilon \rc \\
	\leq \sum_{\sigma \in \{-,+\}} \PP^* \lc \sup_{\theta \in [0, \pi/4]} \labs V_{n,\Delta,1}^\sigma(\theta) - W_{n,k}(C_{p,\theta,1}) \rabs > \epsilon/2 \rc \\
	+ 2 \PP^* \lc \sup_{\theta \in [0, \pi/4]} R_{n,\Delta,1}(\theta) > \epsilon /2 \rc.
\end{multline*}
The first term on the right-hand side vanishes as $n \to \infty$, while the limsup over $n \to \infty$ of the second term can be made arbitrarily small as $\Delta \ra 0$.

Similar computations for the term $i=2$ in~\eqref{eq:expansion_stochastic_term_separately} are to be found in Appendix~F.

\smallskip
\noindent \textbf{The bias term $r_{n,p}(\theta)$.} \quad
\rev{Note that all sets $\hat{C}_{p,\theta}$ for $\theta \in [0, \pi/4]$ are contained in the set $\hat{C}_{p,\pi/4}$ and, with probability tending to one, the latter is contained in the set $\{(x, y) : 0 \le x \wedge y \le M \}$ for $M = M(p) = 2^{1/p}$. It follows that, with probability tending to one,
\begin{align*}
	\lefteqn{
	\sup_{\theta \in [0, \pi/4]} \labs r_{n,p}(\theta) \rabs
	\leq \sqrt{k} \iint_{0 \leq x \wedge y \le M} 
	\labs 
		\frac{k}{n} c \lp \frac{k}{n}x, \frac{k}{n}y \rp - \lambda(x,y) 
	\rabs \d x \d y 
	} \\
	&= \sqrt{k} M \iint_{0 \leq u \wedge v \leq 1} \labs \frac{k}{n} M c \lp \frac{k}{n}M u, \frac{k}{n}M v \rp - \lambda(u,v) \rabs \d u \d v \\
	&= \sqrt{k} M \mathcal{D}_{n/k} \lp \frac{k}{n}M \rp 
	+ \sqrt{k} M \iint_{ \substack{0 \leq u \wedge v \leq 1 \\ u \vee v > \frac{n}{k}} } \frac{k}{n} M c \lp \frac{k}{n}M u, \frac{k}{n}M v \rp \d u \d v \\
	&\quad +  \sqrt{k} M \iint_{ \substack{0 \leq u \wedge v \leq 1 \\ u \vee v > \frac{n}{k}} } \lambda(u,v) \d u \d v.
\end{align*}
By the first part of Assumption~\ref{ass:bias}, the first term in the right-hand side vanishes asymptotically as $n \ra \infty$. The second integral actually equals zero for any value of $n \in \N$ since the copula density $c$ vanishes outside $[0,1]^2$. For the last integral, assuming $n$ large enough so that $n/k > 1$, the domain of integration is the union of two distinct parts given by
$\{(u,v): 0 \leq u \leq 1, v > n/k\}$ and $\{(u,v): u > n/k, 0 \leq v \leq 1\}$. For the first part, we have
\[
	M \sqrt{k} \iint_{\substack{0 \leq u \leq 1 \\ v > n/k}} \lambda(u,v) \d u \d v
	\leq M \sqrt{k} \lp \Phi_p(\tfrac{\pi}{2}) - \Phi_p(\tfrac{\pi}{2}-\tfrac{k}{n}) \rp,
\]
while for the second part, we have
\[
	M \sqrt{k} \iint_{\substack{u > n/k \\ 0 \leq v \leq 1}} \lambda(u,v) \d u \d v
	\leq M \sqrt{k} \Phi_p(\tfrac{k}{n}).
\]
By the second part of Assumption~\ref{ass:bias}, those two quantities converge to zero as $n \ra \infty$.}

\smallskip
\noindent \textbf{The random set term $Z_{n,p}(\theta)$.} \quad
Similarly as for the stochastic term, we work on
\[
	Z_{n,p,i}(\theta) 
	= \sqrt{k} \lp \Lambda \lp \hat{C}_{p,\theta,i} \rp - \Lambda \lp C_{p,\theta,i} \rp \rp
\]
separately for $i = 1,2$. 
In particular, our aim is to show that, as $n \ra \infty$,
\begin{multline}
\label{eq:expansion_random_set_term_1}
	\sup_{\theta \in [0,\pi/4]} \labs Z_{n,p,1}(\theta) - \int_0^{x_p(\theta)} \lambda(x,x\tan\theta) \lacc w_{1n}(x) \tan\theta - w_{2n} \lp x \tan\theta \rp \racc \d x \rabs \\ = \oh_{\PP}(1)
\end{multline}
and, for $p < \infty$ (since this is the most involved case),
\begin{multline}
\label{eq:expansion_random_set_term_2}
	\sup_{\theta \in [0,\pi/4]} \labs Z_{n,p,2}(\theta) - \int_{x_p(\theta)}^\infty \lambda(x,y_p(x)) \lacc y_p'(x) w_{1n}(x) - w_{2n} \lp y_p(x) \rp \racc \d x \rabs 
	\\ = \oh_{\PP}(1).
\end{multline}

\revtwo{We start with~\eqref{eq:expansion_random_set_term_1}. Since, as pointed out earlier, $y_p(x) \geq x\tan\theta$ on $[0,x_p(\theta)]$, we have, on an event with probability tending to one,
\[
	Z_{n,p,1}(\theta) = \sqrt{k} \int_0^{x_p(\theta)} \int_{x\tan\theta}^{x\tan\theta + \tfrac{z_{n,\theta}(x)}{\sqrt{k}}} \lambda(x,y) \d y \d x,
\]
where $z_{n,\theta}$ was defined in~\eqref{eq:def_z_n}.
Changing variables $y(z) = x\tan\theta + z/\sqrt{k}$ and making use of the mean value theorem leads to
\[
	Z_{n,p,1}(\theta) = \int_0^{x_p(\theta)} \lambda \lp x, x\tan\theta + \frac{y(x,\theta)}{\sqrt{k}} \rp z_{n,\theta}(x) \d x,
\]
where $y(x,\theta)$ is between $0$ and $z_{n,\theta}(x)$. Observing that
\begin{align*}
	&\lambda \lp x, x\tan\theta + \frac{y(x,\theta)}{\sqrt{k}} \rp z_{n,\theta}(x) - \lambda(x,x\tan\theta) \lacc w_{1n}(x) \tan\theta - w_{2n} \lp x \tan\theta \rp \racc \\
	&\qquad = \lacc \lambda \lp x, x\tan\theta + \frac{y(x,\theta)}{\sqrt{k}} \rp - \lambda(x,x\tan\theta) \racc z_{n,\theta}(x) \\
	&\qquad \qquad + \left[ z_{n,\theta}(x) - \lacc w_{1n}(x) \tan\theta - w_{2n} \lp x \tan\theta \rp \racc \right] \lambda(x,x\tan\theta),
\end{align*}
we see that sufficient conditions for relation~\eqref{eq:expansion_random_set_term_1} to hold are provided by
\begin{align}
	\sup_{\theta \in [0,\pi/4]} \int_0^{x_p(\theta)} \labs \lambda \lp x, x\tan\theta + \frac{y(x,\theta)}{\sqrt{k}} \rp - \lambda(x,x\tan\theta) \rabs |z_{n,\theta}(x)| \d x &= \oh_{\PP}(1) \label{eq:expansion_random_set_term_1_1} \\
	\sup_{\theta \in [0,\pi/4]} \int_0^{x_p(\theta)} \lambda(x,x\tan\theta) \labs z_{n,\theta}(x) - \lacc w_{1n}(x) \tan\theta - w_{2n} \lp x \tan\theta \rp \racc \rabs \d x &= \oh_{\PP}(1).
	\label{eq:expansion_random_set_term_1_2}
\end{align}}

\revtwo{Let us first consider~\eqref{eq:expansion_random_set_term_1_1}. Since we will need uniform continuity of $\lambda$, we will need to work on a compact domain. Hence, for some $M > 1$, we consider separately the case $\theta \in [0,\arctan(1/M)]$, on which we have no control on the value of $x_p(\theta)$ which can be arbitrary large, and the case $\theta \in [\arctan(1/M), \pi/4]$ on which it is easy to see that $x_p(\theta) \leq 1 + M$. It will be useful to note that as $n \ra \infty$, by~\eqref{eq:z_n_boundedness},
\begin{equation}
\label{eq:z_n_boundedness_2}
	\|z_n\| := \sup_{\theta \in [0,\pi/4]} \sup_{x \in [0,x_p(\theta)]} |z_{n,\theta}(x)| = \Oh_{\PP}(1).
\end{equation}}

\revtwo{Consider the supremum of~\eqref{eq:expansion_random_set_term_1_1} on the region $\theta \in [\arctan(1/M), \pi/4]$ first, for which the domain of integration is included in $x \in [0,1+M]$. We split this domain further on $x \in [0,\delta]$ and $x \in [\delta, 1+M]$ for some $\delta > 0$. The latter case will be considered first. Observe that, by~\eqref{eq:z_n_boundedness_2},
\[
%	\sup_{\theta \in [\arctan(1/M), \pi/4]} \sup_{x \in [\delta,1+M]} 
	\sup_{\substack{\theta \in [\arctan(1/M), \pi/4]\\x \in [\delta,1+M]}}
	\left\| \lp x, x\tan\theta + \frac{y(x,\theta)}{\sqrt{k}} \rp - (x,x\tan\theta) \right\|_1
	\leq \frac{\|z_n\|}{\sqrt{k}} = \oh_{\PP}(1),
\]
which, by uniform continuity of $\lambda$ on $[\delta, 1+M] \times [\delta/(2M), 2+M]$ implies
\[
	\sup_{\theta \in [\arctan(1/M), \pi/4]} \sup_{x \in [\delta,1+M]} \labs \lambda \lp x, x\tan\theta + \frac{y(x,\theta)}{\sqrt{k}} \rp - \lambda(x,x\tan\theta) \rabs = \oh_{\PP}(1).
\]
Consequently,
\[
	\sup_{\theta \in [\arctan(1/M), \pi/4]} \int_\delta^{x_p(\theta)} \labs \lambda \lp x, x\tan\theta + \frac{y(x,\theta)}{\sqrt{k}} \rp - \lambda(x,x\tan\theta) \rabs |z_{n,\theta}(x)| \d x
\]
is $\oh_{\PP}(1)$ too.
The next step is to consider the same region for $\theta$ but $x \in [0,\delta]$. The idea is to bound the integral as follows: by homogeneity of $\lambda$ in~\eqref{eq:lambda_homogen}, we have
\begin{multline*}
	\sup_{\theta \in [\arctan(1/M), \pi/4]} \int_0^\delta \labs \lambda \lp x, x\tan\theta + \frac{y(x,\theta)}{\sqrt{k}} \rp - \lambda(x,x\tan\theta) \rabs |z_{n,\theta}(x)| \d x \\	
	\leq 32 \sup_{y \geq 0} \lambda(1,y) \sup_{\theta \in [\arctan(1/M), \pi/4]} \sup_{x \in [0,\delta]} \frac{|z_{n,\theta}(x)|}{(x\tan\theta)^{1/16}} (\tan\theta)^{1/16} \delta^{1/16}.
\end{multline*}
Using~\eqref{eq:z_n_boundedness} and picking $\delta$ small enough permits to conclude and to show that~\eqref{eq:expansion_random_set_term_1_1} is verified on the region $\theta \in [\arctan(1/M), \pi/4]$.}

\revtwo{Consider now the supremum of~\eqref{eq:expansion_random_set_term_1_1} on the region $\theta \in [0,\arctan(1/M)]$, we split again the domain of integration in two parts : $x \in [0,1]$ and $x \in [1, x_p(\theta)]$. We start with the region near zero. Similar computations as done previously lead to
\begin{multline*}
	\sup_{\theta \in [0, \arctan(1/M)]} \int_0^1 \labs \lambda \lp x, x\tan\theta + \frac{y(x,\theta)}{\sqrt{k}} \rp - \lambda(x,x\tan\theta) \rabs |z_{n,\theta}(x)| \d x \\	
	\leq 32 \sup_{y \geq 0} \lambda(1,y) \sup_{\theta \in [0,\arctan(1/M)]} \sup_{x \in [0,1]} \frac{|z_{n,\theta}(x)|}{(x\tan\theta)^{1/16}} \times (\tan\theta)^{1/16}.
\end{multline*}
Picking $M>1$ large enough permits to conclude for this region. For the region $x \in [1, x_p(\theta)]$, using the homogeneity of $\lambda$ in~\eqref{eq:lambda_homogen} and factoring out $(x\tan\theta)^{1/16}$, we can show that
\begin{multline*}
	\sup_{\theta \in [0, \arctan(1/M)]} \int_1^{x_p(\theta)} \labs \lambda \lp x, x\tan\theta + \frac{y(x,\theta)}{\sqrt{k}} \rp - \lambda(x,x\tan\theta) \rabs |z_{n,\theta}(x)| \d x \\	
	\leq 64 
%	\sup_{\theta \in [0,\arctan(1/M)]} \sup_{x \in [1,x_p(\theta)]}
	\sup_{\substack{\theta \in [0,\arctan(1/M)]\\x \in [1,x_p(\theta)]}}
	\frac{|z_{n,\theta}(x)|}{(x\tan\theta)^{1/16}}
%	\\
	\times 
%	\sup_{\theta \in [0,\arctan(1/M)]} \sup_{|z| \leq \|z_{n}\|} 
	\sup_{\substack{\theta \in [0,\arctan(1/M)]\\|z| \leq \|z_{n}\|}}	
	\lambda(1, \tan\theta + z/\sqrt{k}).
\end{multline*}
By continuity of $\lambda$, we have $\lim_{y \ra 0} \lambda(1,y) = \lambda(1, 0) = 0$: indeed, we have
\begin{align*}
	\infty &> \Lambda([1,\infty) \times [0, 1]) = \iint_{x \ge 1 \ge y \ge 0} \lambda(x,y) \, \d y \, \d x \\
	&= \iint_{x \ge 1 \ge y \ge 0} x^{-1} \lambda(1,y/x) \, \d y \, \d x = \int_{x\ge 1} \int_{z=0}^{1/x} \lambda(1,z) \, \d z \, \d x,
\end{align*}
so that if $\lambda(1, 0) > 0$, the inner integral would of the order $\lambda(1,0)/x$ as $x \to \infty$, making the outer integral diverge, yielding a contradiction. 
Again using~\eqref{eq:z_n_boundedness_2} permits to conclude for the region $\theta \in [0, \arctan(1/M)]$ by picking $M>1$ large enough. Combining everything, we have proven~\eqref{eq:expansion_random_set_term_1_1}.}

\revtwo{We now turn to~\eqref{eq:expansion_random_set_term_1_2}. By arguments similar to those leading up to~\eqref{eq:z_n_poids}, we have
\begin{multline*}
	\sup_{\theta \in [0,\pi/4]} \sup_{x \in [0,x_p(\theta)]} \frac{\labs z_{n,\theta}(x) - \lacc w_{1n}(x) \tan\theta - w_{2n} \lp x \tan\theta \rp \racc \rabs}{(x\tan\theta)^{1/16}} \\
	= \sup_{\theta \in [0,\pi/4]} \sup_{x \in [0,x_p(\theta)]} \frac{\labs v_{2n} \lp x\tan\theta + \frac{w_{1n}(x) \tan\theta}{\sqrt{k}} \rp + w_{2n}(x\tan\theta) \rabs}{(x\tan\theta)^{1/16}} = \oh_{\PP}(1). 
\end{multline*}
It follows from the homogeneity~\eqref{eq:lambda_homogen} and a computation similar to the previous ones that~\eqref{eq:expansion_random_set_term_1_2} holds. As already explained, \eqref{eq:expansion_random_set_term_1_1} and \eqref{eq:expansion_random_set_term_1_2} imply \eqref{eq:expansion_random_set_term_1}.}

Finally, we show~\eqref{eq:expansion_random_set_term_2}. Since $y_p(x) \leq x\tan\theta$ on $[x_p(\theta), \infty]$, we have, on an event with probability tending to one,
\[
	Z_{n,p,2}(\theta) = \sqrt{k} \int_{x_p(\theta)}^\infty \int_{y_p(x)}^{y_p(x) + s_{n,p}(x)/\sqrt{k}} \lambda(x,y) \d y \d x,
\]
where $s_{n,p}$ was defined in~\eqref{eq:def_s_n}.
Changing variables $y(s) = y_p(x) + s/\sqrt{k}$ and making use of the mean value theorem leads to
\[
	Z_{n,p,2}(\theta) = \int_{x_p(\theta)}^\infty \lambda \lp x, y_p(x) + \frac{s(x,p)}{\sqrt{k}} \rp s_{n,p}(x) \d x,
\]
where $s(x,p)$ is between $0$ and $s_{n,p}(x)$. Observing that
\begin{multline*}
	\lambda \lp x, y_p(x) + \frac{s(x,p)}{\sqrt{k}} \rp s_{n,p}(x) - \lambda(x,y_p(x)) \lacc y_p'(x) w_{1n}(x) - w_{2n} \lp y_p(x) \rp \racc \\
	= \lacc \lambda \lp x, y_p(x) + \frac{s(x,p)}{\sqrt{k}} \rp - \lambda(x,y_p(x)) \racc s_{n,p}(x) \\
	+ \left[ s_{n,p}(x) - \lacc y_p'(x) w_{1n}(x) - w_{2n} \lp y_p(x) \rp \racc \right] \lambda(x,y_p(x)),
\end{multline*}
and that for any $\theta \in [0,\pi/4]$ and $1 \leq p < \infty$ we have $x_p(\theta) \geq 2^{1/p}$, we see that sufficient conditions for~\eqref{eq:expansion_random_set_term_2} to hold are provided by
\begin{align}
	\int_{2^{1/p}}^\infty \labs \lambda \lp x, y_p(x) + \frac{s(x,p)}{\sqrt{k}} \rp - \lambda(x,y_p(x)) \rabs |s_{n,p}(x)| \d x &= \oh_{\PP}(1),
	\label{eq:expansion_random_set_term_2_1} \\
	\int_{2^{1/p}}^\infty \lambda(x,y_p(x)) \labs s_{n,p}(x) - \lacc y_p'(x) w_{1n}(x) - w_{2n} \lp y_p(x) \rp \racc \rabs \d x &= \oh_{\PP}(1).
	\label{eq:expansion_random_set_term_2_2}
\end{align}

Let us first consider~\eqref{eq:expansion_random_set_term_2_1}. Since we will need uniform continuity of $\lambda$ again, we want to work on a compact domain. Hence, for some $M > 1$, we consider separately the cases $x \in [2^{1/p},M]$ and $x \in [M,\infty)$. Note that by~\eqref{eq:s_n_poids} in Appendix~F,
\begin{equation}
\label{eq:s_n_boundedness}
	\|s_{n,p}\| := \sup_{x \in [2^{1/p},\infty)} |s_{n,p}(x)| = \Oh_{\PP}(1)
\end{equation}
as $n \ra \infty$, and, by the mean value theorem, for each $x \in [2^{1/p},\infty)$,
\begin{equation}
\label{eq:s_n_MVT}
	s_{n,p}(x) = y_p'(c_n(x)) w_{1n}(x) + v_{2n} \lp y_p \lp x + \frac{w_{1n}(x)}{\sqrt{k}} \rp \rp,
\end{equation}
where $c_n(x)$ is between $x$ and $x + w_{1n}(x)/\sqrt{k}$ and $y_p'(x) = -1/(x^p-1)^{1 + 1/p}$. Since
\[
	\sup_{x \in [2^{1/p},\infty)} \left\| \lp x, y_p(x) + \frac{s(x,p)}{\sqrt{k}} \rp - (x,y_p(x)) \right\|_1 \leq \frac{\|s_{n,p}\|}{\sqrt{k}} = \oh_{\PP}(1),
\]
it follows from the uniform continuity of $\lambda$ on compact sets included in $[0,\infty)^2 \setminus \{(0,0)\}$ that for any $M>1$,
\[
	\sup_{x \in [2^{1/p},M]} \labs \lambda \lp x, y_p(x) + \frac{s(x,p)}{\sqrt{k}} \rp - \lambda(x,y_p(x)) \rabs = \oh_{\PP}(1)
\]
and~\eqref{eq:expansion_random_set_term_2_1} follows easily for the region $x \in [2^{1/p},M]$. For the region  $x \in [M,\infty]$, just observe that by~\eqref{eq:s_n_boundedness},
\begin{multline}
\label{eq:integrals}
	\int_M^\infty \labs \lambda \lp x, y_p(x) + \frac{s(x,p)}{\sqrt{k}} \rp - \lambda(x,y_p(x)) \rabs |s_{n,p}(x)| \d x \\
	\leq \|s_{n,p} \| \lacc \int_M^\infty \lambda \lp x, y_p(x) + \frac{s(x,p)}{\sqrt{k}} \rp \d x + \int_M^\infty \lambda(x,y_p(x)) \d x \racc.
\end{multline}
We first consider the second integral which is easier to deal with. By homogeneity~\eqref{eq:lambda_homogen} and a change of variables,
\begin{align*}
	&\int_M^\infty \lambda(x,y_p(x)) \d x 
	= \int_M^\infty \lambda((x^p-1)^{1/p}, 1) (y_p(x))^{-1} \d x \\
	&\quad \leq \int_{(M^p-1)^{1/p}}^\infty \lambda(u,1) \lp \frac{u^p}{u^p + 1} \rp^{1-1/p} \d u
	\leq \int_{(M^p-1)^{1/p}}^\infty \lambda(u,1) \d u.
\end{align*}
Since $\int_0^\infty \lambda(u,1) \d u = 1$, we may apply the dominated convergence theorem to ensure that the above integral can be made as small as desired provided we pick $M>1$ large enough. Now consider the second integral on the right-hand side of~\eqref{eq:integrals}. Since $0 < s(x,p) < s_{n,p}(x)$ or $0 > s(x,p) > s_{n,p}(x)$, we have, by~\eqref{eq:s_n_poids}, for all $M>2^{1/p}$,
\[
	\int_M^\infty \lambda \lp x, y_p(x) + \frac{s(x,p)}{\sqrt{k}} \rp \d x 
	=  \int_M^\infty \lambda(x,y_p(x)) \d x + \oh_{\PP}(1),
\]
and since the integral on the right-hand side vanishes as $M \ra \infty$, we have shown~\eqref{eq:expansion_random_set_term_2_1}.

Let us now consider~\eqref{eq:expansion_random_set_term_2_2}. Using~\eqref{eq:marginal_weak_convergences},~\eqref{eq:quantile_weak_convergences}, the continuity of $y_p$ and its derivative $y_p'$, we observe that
\begin{multline*}
	\sup_{x \in [2^{1/p}, \infty)} \labs s_{n,p}(x) - \lacc y_p'(x) w_{1n}(x) - w_{2n}(y_p(x)) \racc \rabs \\
	\leq \sup_{x \in [2^{1/p}, \infty)} 
	\Big| \big( y_p'(c_n(x)) - y_p'(x) \big) w_{1n}(x)
	+ \big( v_{2n} \big( y_p (x + \tfrac{w_{1n}(x)}{\sqrt{k}}) \big) + w_{2n}(y_p(x)) \big) \Big|,
\end{multline*}
which is $\oh_{\PP}(1)$.
By integrability of the function $x \in [2^{1/p}, \infty) \mapsto \lambda(x,y_p(x))$, it follows that
\[
	\int_{2^{1/p}}^\infty \lambda(x,y_p(x)) \labs s_{n,p}(x) - \lacc y_p'(x) w_{1n}(x) - w_{2n} \lp y_p(x) \rp \racc \rabs \d x = \oh_{\PP}(1),
\]
and~\eqref{eq:expansion_random_set_term_2_2} holds.

\smallskip
\noindent \textbf{Conclusion of the proof of Theorem~\ref{thm:expansionEAM}.} \quad
Adding the asymptotic expansions of the three terms in \eqref{eq:threeterms} permits to conclude the proof for $\theta \in [0,\pi/4]$. By the symmetry argument in Appendix~E, we can extend the conclusion to all $\theta \in [0,\pi/2]$. The proof of Theorem~\ref{thm:expansionEAM} is complete.

\medskip
\noindent \rev{\textbf{Acknowledgments.}} \quad
\rev{Stéphane Lhaut was financially supported by the Fonds de la Recherche Scientifique – FNRS through a FRIA grant. The authors would like to thank two anonymous Referees for helpful comments and suggestions on an earlier version of the paper.}

%%%%%%%%%%%%%%%%%%%%%%%%%%%%%%%%%%%%%%%%%%%%%%%%%%%%%%%%%%%%%%%%%%%%%%%%
\bibliographystyle{acm}
\bibliography{references.bib}
%%%%%%%%%%%%%%%%%%%%%%%%%%%%%%%%%%%%%%%%%%%%%%%%%%%%%%%%%%%%%%%%%%%%%%%%

\newpage
\appendix

\renewcommand{\thesection}{\Alph{section}} 
\numberwithin{equation}{section}

\begin{center}
\Huge \textbf{Supplement}
\end{center}

Appendix~\ref{app:link} provides the link between the densities of the tail copula measure $\Lambda$ and the angular measure $\Phi_p$. Appendices~\ref{app:proofProp1} and~\ref{app:proofProp2} contain the proofs of Propositions~\ref{prop:general_weak_convergence} and~\ref{prop:weighted} in the paper, respectively. Appendix~\ref{app:proofCor2} spells out the proof of Corollary~\ref{cor:quantile}. Finally, Appendices~\ref{app:proofSym} and~\ref{app:proofi2} provide further details on certain parts of the proof of Theorem~\ref{thm:expansionEAM}.

%%%%%%%%%%%%%%%%%%%%%%%%%%%%%%%%%%%%%%%%%%%%%%%%%%%%%%%%%%%%%%%%%%%%%%%%
\section{Link between densities of $\Lambda$ and $\Phi_p$}
\label{app:link}
%%%%%%%%%%%%%%%%%%%%%%%%%%%%%%%%%%%%%%%%%%%%%%%%%%%%%%%%%%%%%%%%%%%%%%%%

For any $x,y > 0$, we have
\[
\lambda(x,y) = \frac{\partial^2}{\partial x \partial y} \Lambda \lp (0,x] \times (0,y] \rp = \frac{\partial^2}{\partial x \partial y} \nu \lp [x^{-1},\infty) \times [y^{-1},\infty) \rp,
\]
where $\nu$ factorizes in polar coordinates under $T$ in~\eqref{eq:polar_nu}, i.e., $T \# \nu = \nu_{-1} \otimes \Phi_p$, so that 
\begin{align*}
	&\nu \lp [x^{-1},\infty) \times [y^{-1},\infty) \rp \\
	%	&= T \# \nu \lp T([x^{-1},\infty) \times [y^{-1},\infty)) \rp \\
	&= \lp \nu_{-1} \otimes \Phi_p \rp \lp \lacc T(u,v): u \geq x^{-1}, v \geq y^{-1}  \racc  \rp \\
	&= \lp \nu_{-1} \otimes \Phi_p \rp \lp \lacc (r,\theta): \frac{r\sin\theta}{\|(\sin\theta,\cos\theta)\|_p} \geq x^{-1}, \frac{r\cos\theta}{\|(\sin\theta,\cos\theta)\|_p} \geq y^{-1}  \racc \rp \\
	%	&= \nu_{-1} \otimes \Phi_p \lp \lacc (r,\theta): r \geq \frac{\|(\sin\theta,\cos\theta)\|_p}{x \sin \theta} \vee \frac{\|(\sin\theta,\cos\theta)\|_p}{y \cos \theta} \racc \rp \\
	&= \int_0^{\pi/2} \lp \frac{x\sin\theta}{\|(\sin\theta,\cos\theta)\|_p} \wedge \frac{y\cos\theta}{\|(\sin\theta,\cos\theta)\|_p} \rp \varphi_p(\theta) \d\theta \\
	&= x \int_0^{\arctan(y/x)} \frac{\sin \theta}{\|(\sin\theta,\cos\theta)\|_p} \varphi_p(\theta) \d\theta \\
	&\qquad + y \int_{\arctan(y/x)}^{\pi/2} \frac{\cos \theta}{\|(\sin\theta,\cos\theta)\|_p} \varphi_p(\theta) \d\theta. 
\end{align*}
Differentiating this expression with respect to $x$ and $y$ using the Leibniz differentiation rule yields the desired formula~\eqref{eq:densities_relation} for $\lambda$, showing that its continuity on $(0,\infty)^2$ is equivalent to the one of $\varphi_p$ on $(0,\pi/2)$.

%%%%%%%%%%%%%%%%%%%%%%%%%%%%%%%%%%%%%%%%%%%%%%%%%%%%%%%%%%%%%%%%%%%%%%%%
\section{Proof of Proposition~\ref{prop:general_weak_convergence}}
\label{app:proofProp1}
%%%%%%%%%%%%%%%%%%%%%%%%%%%%%%%%%%%%%%%%%%%%%%%%%%%%%%%%%%%%%%%%%%%%%%%%

The proof is an application of Theorem~19.28 in~\cite{VVV1998} which we recall here for convenience.
\rev{
\begin{theorem*}\itshape
Let $X_1,\ldots,X_n$ be a random sample from a probability distribution $P$ on a measurable space $(S, \mathcal{S})$. Let $P_n$ denote the associated empirical distribution. Let $\cF_n = \{f_{n,t}: S \ra \R,  \, t \in T\}$ be a class of measurable functions indexed by a totally bounded semimetric space $(T,\rho)$ satisfying
\[
	\lim_{n \ra \infty} \sup_{\rho(s,t) < \delta_n} P \lp ( f_{n,s} - f_{n,t} )^2 \rp = 0, \qquad \forall \delta_n \downarrow 0,
\]
and with envelope function $F_n$ satisfying the Lindeberg condition
\[
	P \lp F_n^2 \rp = \Oh(1) \qquad \text{and} \qquad P \lp F_n^2 \, \1\{F_n > \epsilon \sqrt{n}\} \rp = \oh(1), \quad \epsilon > 0.
\]
If every class $\cF_n$ is suitably measurable in the sense of Example~2.3.4 in~\cite{VVV1996} and if
\[
	\lim_{n \ra \infty} \sup_{Q} \int_{0}^{\delta_n} \sqrt{\log N(\epsilon \|F_n\|_{Q,2}, \cF_n, L_2(Q))} \d\epsilon = 0, \qquad \forall \delta_n \downarrow 0,
\]
where the supremum is taken over all finitely discrete probability measures $Q$ on $(S, \mathcal{S})$ for which $\cF_n$ is not identically zero and $N$ is the covering number as defined in~\cite[page~274]{VVV1998}, then the sequence of empirical processes $\{G_n(f_{n,t}) := \sqrt{n}(P_n(f_{n,t}) - P(f_{n,t})): t \in T\}$ converges in distribution to a tight Gaussian process, provided the sequence of covariance functions $P(f_{n,s}f_{n,t}) - P(f_{n,s})P(f_{n,t})$ converges pointwise on $T \times T$.
\end{theorem*}
}

\begin{proof}[Proof of Proposition~\ref{prop:general_weak_convergence}]
For any $n \in \N$ and $A \in \cA$, let $\cF_n$ be the class of functions $f_{n,A} : \R^d \ra \R$ defined by
\[
	f_{n,A}(x) := \sqrt{\tfrac{n}{k}} \1 \lacc x \in \tfrac{k}{n} A \racc.
\]
\rev{These classes satisfy the measurability assumptions given the assumptions on the class $\cA$.}
Note that for $A \in \cA$,
\[
	G_n \lp f_{n,A} \rp 
	= \sqrt{k} \lp \tfrac{n}{k} P_n ( \tfrac{k}{n} A ) - \tfrac{n}{k} P ( \tfrac{k}{n} A ) \rp.
\]
Let us verify that the technical conditions underlying Theorem~19.28 in~\cite{VVV1998} are satisfied.

We first show that the metric space $(\cA, \rho)$ is totally bounded. Indeed, noting that $\Lambda(L_M) \leq d \times M$, since each element of $\cA$ is included in $L_M$, we may construct the standard empirical process with respect to the probability measure $P_M(\cdot) := \Lambda(\cdot \cap L_M) / \Lambda(L_M)$ on $L_M$,
indexed by functions $f_A := \1_A$ for $A \in \cA$, and use the fact that $\cA$ is a VC-class with $|f_A| \leq F_M := \1_{L_M}$ pointwise for each $A \in \cA$. Apply Theorem~19.14 in \cite{VVV1998} to conclude that the considered process converges weakly in $\ell^\infty(\cA)$, where we identify $\cA$ with its class of indicator functions. In view of Theorem~18.14 and Lemma~18.15 in~\cite{VVV1998}, the metric space $(\cA, \rho_M)$ with $\rho_M := \rho/\Lambda(L_M)$ must be totally bounded, which also guarantees that $(\cA, \rho)$ is totally bounded.

We now show that
\begin{equation}
	\label{eq:rho_relation_with_L_2}
	\lim_{n \ra \infty} \sup_{A,A' \in \cA: \rho(A,A') < \delta_n} P \lp (f_{n,A} - f_{n,A'})^2 \rp = 0
\end{equation}
\mdf{whenever $\delta_n \downarrow 0$ as $n \to \infty$.}
Indeed,
\begin{align*}
%	\lefteqn{
		\sup_{A,A': \rho(A,A') < \delta_n} P \lp (f_{n,A} - f_{n,A'})^2 \rp
%	} \\
	&= \sup_{A,A': \rho(A,A') < \delta_n} \tfrac{n}{k} P \lp (\1_{(k/n) A} - \1_{(k/n) A'})^2 \rp \\
	&= \sup_{A,A': \rho(A,A') < \delta_n} \tfrac{n}{k} P \lp \tfrac{k}{n} (A \symdif A') \rp.
\end{align*}
Fix $\epsilon > 0$. It follows from our assumptions that there exists $N_1(\epsilon) \in \N$ such that for $n \geq N_1(\epsilon)$,
\[
\sup_{A,A' \in \cA: \rho(A,A') < \delta_n} \labs \tfrac{n}{k} P \lp \tfrac{k}{n} (A \symdif A') \rp - \Lambda(A \symdif A') \rabs \leq \epsilon/2,
\]
while, since $\delta_n \ra 0$ as $n \ra \infty$, we may find $N_2(\epsilon) \in \N$ such that for $n \geq N_2(\epsilon)$,
\[
\sup_{A,A' \in \cA: \rho(A,A') < \delta_n} \Lambda(A \symdif A') \leq \epsilon/2.
\] 
It follows that for $n \geq \max\{N_1(\epsilon),N_2(\epsilon)\}$, we have
\[
\sup_{A,A' \in \cA: \rho(A,A') < \delta_n} P \lp (f_{n,A} - f_{n,A'})^2 \rp \leq \epsilon,
\]
which yields~\eqref{eq:rho_relation_with_L_2}.

We now verify that the classes $\cF_n$ possess envelope functions $F_n$ satisfying the Lindeberg condition. It suffices to note that for every $n \in \N$ and $A \in \cA$,
\[
f_{n,A}(x) 
= \sqrt{\tfrac{n}{k}} \1 \lacc x \in \tfrac{k}{n} A \racc
\leq \sqrt{\tfrac{n}{k}} \1 \lacc x \in \tfrac{k}{n} L_M \racc
=: F_n(x).
\]
The first part of the Lindeberg condition clearly holds since
\[
\lim_{n \ra \infty} P \lp F_n^2 \rp
= \lim_{n \ra \infty} \tfrac{n}{k} P \lp \tfrac{k}{n} L_M \rp
= \Lambda \lp L_M \rp \leq d \cdot M.
\]
The second part of the Lindeberg condition is immediate since for any $\epsilon > 0$, the indicator $\1\{F_n > \epsilon \sqrt{n}\}$ is zero as soon as $1/\sqrt{k} \leq \epsilon$, which is the case for $n$ large enough, since, by assumption, $k = k(n) \ra \infty$. 

The last condition to verify is the one associated to the uniform entropy integral. Here, since the VC dimension of any $\cF_n$ equals the one of $\cA$, which is finite by assumption, we may simply apply~\cite[Example~2.11.24]{VVV1996}.

Consequently, Theorem~19.28 in~\cite{VVV1996} applies and the covariance of the limiting Gaussian process is obtained by computing the limit
\[
\lim_{n \ra \infty} P \lp f_{n,A} f_{n,A'} \rp - P \lp f_{n,A} \rp P \lp f_{n,A'} \rp = \Lambda(A \cap A'),
\]
for $A, A' \in \cA$, which follows from our assumption.

Finally, we show~\eqref{eq:asy_equi}. By Example~1.5.10 in~\cite{VVV1996} the processes $\Wnk$ are asymptotically equicontinuous in probability with respect to the standard deviation semi-metric $\rho_2(A, A') := (\EE[ |W_\Lambda(A) - W_\Lambda(A')|^2])^{1/2}$. Since
\[
\EE[ |W_\Lambda(A) - W_\Lambda(A')|^2] = \Lambda(A) + \Lambda(A') - 2 \Lambda(A \cap A')
= \Lambda(A \symdif A') = \rho(A,A'), 
\]
the said equicontinuity property holds with respect to $\rho$ as well.
\end{proof}

%%%%%%%%%%%%%%%%%%%%%%%%%%%%%%%%%%%%%%%%%%%%%%%%%%%%%%%%%%%%%%%%%%%%%%%%
\section{Proof of Proposition~\ref{prop:weighted}}
\label{app:proofProp2}
%%%%%%%%%%%%%%%%%%%%%%%%%%%%%%%%%%%%%%%%%%%%%%%%%%%%%%%%%%%%%%%%%%%%%%%%

\rev{For any $n \in \N$ and $t \in T:= [0,\infty)$, let $\cF_n$ be the class of functions $f_{n,t}: [0,\infty) \ra \R$ defined by
\[
	f_{n,t}(x) := \sqrt{\tfrac{n}{k}} \1 \lp x \leq \tfrac{k}{n}t \rp w(t) \1 \lp t \leq \tfrac{n}{k} \rp.
\]
These classes trivially satisfy the measurability assumptions underlying Theorem~19.28 in~\cite{VVV1998} since it suffices to consider positive rational coordinates $t$.
For each $t \in T$,
\[
	G_n(f_{n,t}) = \sqrt{n} \lp P_{n}(f_{n,t}) - P(f_{n,t}) \rp = w_{n}(t)w(t),
\]
where $P_{n}$ is the empirical distribution of the i.i.d.\ sample $U_{1}, \ldots, U_{n}$ whose common law is $P$, the uniform distribution on $[0,1]$.}

\rev{Our aim is to apply Theorem~19.28 in~\cite{VVV1998}. We first need to find a metric $\rho$ on $T$ such that the metric space $(T,\rho)$ is totally bounded and
\begin{equation}
\label{eq:Pfnsfnt}
	\lim_{n \ra \infty} \sup_{\rho(s,t) < \delta_n} P \lp (f_{n,s} - f_{n,t})^2 \rp = 0,
\end{equation}
for any $(\delta_n)_{n \in \N}$ such that $\delta_n \downarrow 0$. Since the weak limit of the process is expected to be centered Gaussian process, we may choose $\rho$ to be the ``standard deviation metric'' as in~\cite[Lemma~18.15]{VVV1998}. After some computations, we find that it is given by
\[
	\rho(s,t) := \sqrt{w(t)^2 t + w(s)^2 s - 2w(t)w(s) (s \wedge t)}, \qquad s,t \in T. 
\]
With this metric, it is easily checked that \eqref{eq:Pfnsfnt} holds and that the metric space $(T,\rho)$ is totally bounded.}

\rev{We now verify that the classes $\cF_n$ possess envelope functions $F_n$ satisfying the Lindeberg condition. It suffices to note that for every $n \in \N$ and $t \in T$,
\[
	f_{n,t}(x) 
	\leq 
	\sqrt{\tfrac{n}{k}} w \lp \tfrac{n}{k} x \rp \1(x \leq 1) 
	=: F_n(x).
\]
We have
\begin{align*}
	P \lp F_n^2 \rp 
	&= \frac{n}{k} \int_0^1 w \lp \frac{n}{k} x \rp^2 \d x \\
	&= \frac{n}{k} \left\{  \int_0^{k/n}  w \lp \frac{n}{k} x \rp^2 \d x + \int_{k/n}^1 w \lp \frac{n}{k} x \rp^2 \d x \right\} \\
	&= \frac{1}{1-2\delta} + \frac{1- \lp \tfrac{k}{n} \rp^{2\eta-1}}{2\eta-1} \rightarrow \frac{1}{1-2\delta} + \frac{1}{2\eta-1}
\end{align*}
as $n \ra \infty$, so that the first part of the Lindeberg condition holds. 
To verify the second part, we note that for any $\epsilon > 0$,
\begin{align*}
	P \lp F_n^2 \1(F_n > \epsilon \sqrt{n}) \rp
	&= \tfrac{n}{k} \int_0^1  w \lp \tfrac{n}{k} x \rp^2 \1 \lp w \lp \tfrac{n}{k} x \rp > \epsilon \sqrt{k} \rp \d x \\
	&\leq \int_0^{\infty} w(y)^2 \1 \lp w(y) > \epsilon \sqrt{k} \rp \d y.
\end{align*}
Since $y \in (0,\infty) \mapsto w(y)^2$ is integrable and the integrand, which is bounded pointwise by this function, converges pointwise to zero, the dominated convergence theorem permits to conclude.}

\rev{The last condition to verify is the one associated to the uniform entropy integral. We propose to use Example~2.11.24 in~\cite{VVV1996}. To this end, we need to verify that the classes $\cF_n$ have finite VC dimensions which can be bounded independently of $n \in \N$. The VC dimension of $\cF_n$ is, by definition, the VC dimension of the class of sets
\[
	\cS_n := \left\{ S_{n,t} \subseteq \mathbb{R}^2 : t \in T \right\},
\]
where $S_{n,t}$ is the sub-graph of $f_{n,t}$ given by
\[
	S_{n,t} := \left\{ (x,h) : f_{n,t}(x) > h \right\}.
\]
We compute that
\[
	S_{n,t} = \lp \left[ 0, \tfrac{k}{n}t \right] \times \left[  0, \sqrt{\tfrac{n}{k}}w(t) \right] \rp \cup \lp [0,\infty) \times (-\infty,0) \rp.
\]
Hence, each element of the class $\cS_n$ is formed from the union of two rectangles where the second rectangle does not depend on $t$. Hence, the VC dimension of $\cS_n$ is bounded by $4$, which corresponds to the VC dimension of the class of all rectangles on $\R^2$. The uniform entropy condition is then verified.}

\rev{Theorem~19.28 in~\cite{VVV1998} applies and the computation of the covariance function of the limiting Gaussian process is an easy exercise.}

\rev{To prove the last assertion, we need, given any $\epsilon, \eta > 0$, to find $N = N(\epsilon,\eta) \in \N$ such that for any $n \geq N$,
\[
	\PP \lc \sup_{x \in [0,M]} \frac{\labs v_{n}(x) + w_{n}(x) \rabs}{x^\delta} > \eta \rc < \epsilon.
\]
To do so, observe that for any $0 < a < 1$,
\begin{multline}
\label{eq:probas_split}
	\PP \lc \sup_{x \in [0,M]} \frac{\labs v_{n}(x) + w_{n}(x) \rabs}{x^\delta} > \eta \rc
	\leq \PP \lc \sup_{x \in [a,M]} \labs v_{n}(x) + w_{n}(x) \rabs > \eta a^\delta \rc \\
	+ \PP \lc \sup_{x \in [0,a]} \frac{|w_{n}(x)|}{x^\delta} > \frac{\eta}{2} \rc + \PP \lc \sup_{x \in [0,a]} \frac{|v_{n}(x)|}{x^\delta} > \frac{\eta}{2} \rc.
\end{multline}
In view of~\eqref{eq:quantile_weak_convergences}, we may find $N_1 \in \N$ such that for $n \geq N_1$, the first term on the right hand side of~\eqref{eq:probas_split} satisfies
\[
	\PP \lc \sup_{x \in [a,M]} \labs v_{n}(x) + w_{n}(x) \rabs > \eta a^\delta \rc < \frac{\epsilon}{3}.
\]
For the second term on the right hand side of~\eqref{eq:probas_split}, note that by the first part of the proposition, there exists $N_2 \in \N$ such that for $n \geq N_2$,
\[
	\PP \lc \sup_{x \in [0,a]} \frac{|w_{n}(x)|}{x^\delta} > \frac{\eta}{2} \rc
	\leq \PP \lc \sup_{x \in [0,a]} \frac{|W_j(x)|}{x^\delta} > \frac{\eta}{2} \rc + \frac{\epsilon}{6}
\]
and we may find $a_1 \in (0,1)$ small enough such that for $0 < a < a_1$,
\[
	\PP \lc \sup_{x \in [0,a]} \frac{|W_j(x)|}{x^\delta} > \frac{\eta}{2} \rc < \frac{\epsilon}{6}.
\]
For the second term on the right~hand side of~\eqref{eq:probas_split}, note that for $x \in [0,(2k)^{-1}]$, we have $v_n(x) = -x\sqrt{k}$. Hence,
\begin{equation*}
	\sup_{x \in [0, (2k)^{-1}]} \frac{|v_{n}(x)|}{x^\delta}
	= \sqrt{k} \sup_{x \in [0, (2k)^{-1}]} x^{1-\delta}
	\le k^{1/2-1+\delta} \to 0, \qquad n \to \infty.
\end{equation*}
Therefore, we may restrict the supremum to $x \in ((2k)^{-1},a]$. On that region, we have
\begin{equation}
	\label{eq:QequivGam}
	\forall x > (2k)^{-1}, 
	\forall u \ge 0, \qquad
	Q_n(\tfrac{k}{n}x) > u
	\iff
	\Gamma_n(u) < \tfrac{k}{n}x.
\end{equation}
By definition of the absolute value and by the union bound, for every $\eta > 0$,
\begin{multline}
	\label{eq:Psupinf}
	\PP \lc
	\sup_{x \in ((2k)^{-1}, a]}
	\frac{|v_{n}(x)|}{x^\delta} > \eta
	\rc
	\\
	\le
	\PP \lc
	\sup_{x \in ((2k)^{-1}, a]}
	\frac{v_{n}(x)}{x^\delta} > \eta
	\rc
	+
	\PP \lc
	\inf_{x \in ((2k)^{-1}, a]}
	\frac{v_{n}(x)}{x^\delta} < -\eta
	\rc.
\end{multline}
We will treat each of the two terms on the right-hand side of~\eqref{eq:Psupinf} separately. Recall $w_n(x) = \sqrt{k} \lp \tfrac{n}{k} \Gamma_n(\tfrac{k}{n}x) - x \rp$ for $x \ge 0$. 
First, by \eqref{eq:QequivGam}, for $x > (2k)^{-1}$,
\begin{align*}
	\frac{v_{n}(x)}{x^\delta} > \eta
	&\iff Q_{n}(\tfrac{k}{n}x) > \tfrac{k}{n} ( x + \eta x^\delta / \sqrt{k} ) \\
	&\iff 
	\Gamma_n \lp \tfrac{k}{n} (x + \eta x^\delta/\sqrt{k}) \rp
	< \tfrac{k}{n} x \\
	&\iff
	\sqrt{k} \lacc \tfrac{n}{k} \Gamma_n \lp \tfrac{k}{n} (x + \eta x^\delta/\sqrt{k}) \rp - (x + \eta x^\delta/\sqrt{k}) \racc
	< - \eta x^\delta \\
	&\iff - \frac{w_n(x + \eta x^\delta/\sqrt{k})}{x^\delta} > \eta.
\end{align*}
Therefore,
\[
\PP \lc
\sup_{x \in ((2k)^{-1}, a]} \frac{v_n(x)}{x^\delta} > \eta
\rc
=
\PP \lc
\sup_{x \in ((2k)^{-1}, a]} \frac{-w_n(x+\eta x^\delta/\sqrt{k})}{x^\delta} > \eta
\rc.
\]
Let $\gamma = \delta / (\delta + 1/2)$ and note that $\delta < \gamma < 1/2$ since $0 < \delta < 1/2$. We have
\[
\frac{w_n(x+\eta x^\delta/\sqrt{k})}{x^\delta}
= \frac{w_n(x+\eta x^\delta/\sqrt{k})}{(x+\eta x^\delta/\sqrt{k})^\gamma} \cdot \frac{(x+\eta x^\delta/\sqrt{k})^\gamma}{x^\delta}.
\]
Since $\delta / \gamma = \delta + 1/2$, the second factor on the right-hand side satisfies
\begin{align*}
	\frac{(x+\eta x^\delta/\sqrt{k})^\gamma}{x^\delta}
	&= \lp x \cdot x^{-\delta/\gamma} + \eta x^\delta \cdot x^{-\delta/\gamma} / \sqrt{k} \rp^{\gamma} \\
	&= \lp x^{1 - \delta - 1/2} + \eta x^{-1/2} / \sqrt{k} \rp^\gamma
	\le 2, \qquad \forall x \in ((2k)^{-1}, 1].
\end{align*}
We find
\begin{align*}
	\PP \lc
	\sup_{x \in ((2k)^{-1}, a]} 
	\frac{v_n(x)}{x^\delta} 
	> \eta
	\rc
	&\le 
	\PP \lc
	\sup_{x \in ((2k)^{-1}, a]} 
	\frac{w_n(x + \eta x^\delta / \sqrt{k})}{(x + \eta x^\delta / \sqrt{k})^\gamma}
	> \eta / 2
	\rc \\
	&=
	\PP \lc
	\sup_{z \in [0, a + \eta a^\delta / \sqrt{k}]}
	\frac{w_n(z)}{z^\gamma}
	> \eta / 2
	\rc \\
	&\le
	\PP \lc
	\sup_{z \in [0, 2a]}
	\frac{w_n(z)}{z^\gamma}
	> \eta / 2
	\rc
\end{align*}
for $k$ sufficiently large.
As $0 < \gamma < 1/2$, it follows from previous considerations that we may find $N_3 \in \N$ and $a_2 \in (0,1)$ such that for $n \geq \N_3$ and $0 < a < a_2$,
\[
	\PP \lc
	\sup_{z \in [0, 2a]}
	\frac{w_n(z)}{z^\gamma}
	> \eta / 2
	\rc
	< \frac{\epsilon}{6}.
\]
Let us now consider the second term in~\eqref{eq:Psupinf}. By \eqref{eq:QequivGam}, for $x > (2k)^{-1}$ such that also $x \ge \eta x^\delta \sqrt{k}$,
\begin{align*}
	\frac{v_{n}(x)}{x^\delta} \le -\eta
	&\iff Q_{n}(\tfrac{k}{n}x) \le \tfrac{k}{n} ( x - \eta x^\delta / \sqrt{k} ) \\
	&\iff \Gamma_n \lp \tfrac{k}{n} (x - \eta x^\delta/\sqrt{k}) \rp \ge \tfrac{k}{n} x \\
	&\iff
	\sqrt{k} \lacc \tfrac{n}{k} \Gamma_n \lp \tfrac{k}{n} (x - \eta x^\delta/\sqrt{k}) \rp - (x - \eta x^\delta/\sqrt{k}) \racc
	\ge \eta x^\delta \\
	&\iff \frac{w_n(x - \eta x^\delta / \sqrt{k})}{x^\delta} \ge \eta.	
\end{align*}
If $x < \eta x^\delta \sqrt{k}$, then the inequality $v_n(x)/x^\delta \le -\eta$ is impossible, since $Q_n \ge 0$ and thus $v_n(x) \ge -x\sqrt{k}$. It follows that
\begin{align*}
	\PP \lc
	\inf_{x \in ((2k)^{-1}, a]} \frac{v_{n}(x)}{x^\delta} < -\eta
	\rc
	&\le
	\PP \lc
	\inf_{x \in ((2k)^{-1}, a]} \frac{v_{n}(x)}{x^\delta} \le -\eta
	\rc \\
	&=
	\PP \lc
	\sup_{\substack{x \in ((2k)^{-1}, a] \\ x \ge \eta x^\delta \sqrt{k}}} \frac{w_n(x - \eta x^\delta / \sqrt{k})}{x^\delta} \ge \eta
	\rc.
\end{align*}
Since $x - \eta x^\delta / \sqrt{k} \le x$, the weight $1/x^\delta$ can only increase if we replace it by $1/(x - \eta x^\delta / \sqrt{k})^\delta$. But then
\[
\sup_{\substack{x \in ((2k)^{-1}, a] \\ x \ge \eta x^\delta \sqrt{k}}} \frac{w_n(x - \eta x^\delta / \sqrt{k})}{x^\delta}
\le
\sup_{z \in [0,a]} \frac{w_n(z)}{z^\delta}.
\]
Consequently, it follows again from our results on the weighted tail empirical process that we may find $N_4 \in \N$ and $a_3 \in (0,1)$ such that for $n \geq \N_4$ and $0 < a < a_3$,
\[
	\PP \lc
	\inf_{x \in ((2k)^{-1}, a]}
	\frac{v_{n}(x)}{x^\delta} < -\eta
	\rc
	< \frac{\epsilon}{6}.
\]
Combining everything, we have for $n \geq \max\{N_1,\ldots,N_4\}$ and $0 < a < \min\{a_1,a_2,a_3\}$ in~\eqref{eq:probas_split},
\[
	\PP \lc \sup_{x \in [0,M]} \frac{\labs v_{n}(x) + w_{n}(x) \rabs}{x^\delta} > \eta \rc < \epsilon.
\]
This concludes the proof.
}

%%%%%%%%%%%%%%%%%%%%%%%%%%%%%%%%%%%%%%%%%%%%%%%%%%%%%%%%%%%%%%%%%%%%%%%%
\section{Proof of Corollary~\ref{cor:quantile}}
\label{app:proofCor2}
%%%%%%%%%%%%%%%%%%%%%%%%%%%%%%%%%%%%%%%%%%%%%%%%%%%%%%%%%%%%%%%%%%%%%%%%

Let $[v,w] \subset (0,1)$ and consider the map
\[
	\psi: F \in \mathbb{D}_\psi \subset \ell^\infty([0,\pi/2]) \mapsto \psi(F) := F^{-1} \in \ell^\infty\bigl([v,w]\bigr)
\]
where $\mathbb{D}_\psi$ is the set of cumulative distribution functions supported on $[0,\pi/2]$, vanishing at $0$, which are continuously differentiable on $(0,\pi/2)$ with strictly positive derivative. By point (i) of~\cite[Lemma~21.4]{VVV1998}, $\psi$ is Hadamard differentiable tangentially to $C([0,\pi/2])$, the set of continuous functions on $[0,\pi/2]$. Furthermore, its derivative at $F \in \mathbb{D}_\psi$ in the direction $\gamma \in C([0,\pi/2])$ is
\[
	\psi'_F(\gamma) = - \frac{\gamma}{f} \circ F^{-1},
\]
where $f$ is the derivative of $F$. Consequently, the map $\psi$ is Hadamard differentiable at $Q_p \in \mathbb{D}_\psi$, and its derivative, $\psi'_{Q_p}$, is defined and continuous on the whole space $\ell^\infty([0,\pi/2])$. Therefore, it follows from~\cite[Theorem~20.8]{VVV1998} that
\[
	\sup_{u \in [v,w]} \labs \sqrt{k} \lp \hQ_p^{-1}(u) - Q_p^{-1}(u) \rp - \psi'_{Q_p} \lp  \sqrt{k} \bigl( \hQ_p - Q_p \bigr) \rp(u) \rabs = \oh_{\PP}(1),
\]
as $n \ra \infty$. \rev{Corollary~\ref{cor:eapm}, in combination with continuity of the derivative $\psi'_{Q_p}$ implies
\begin{multline*}
	\sup_{u \in [v,w]} \Biggl| \psi'_{Q_p} \lp  \sqrt{k} \bigl( \hQ_p - Q_p \bigr) \rp(u) \\ 
	- \psi'_{Q_p} \lp \frac{E_{n,p}(\cdot) \Phi_p(\tfrac{\pi}{2}) - \Phi_p(\cdot) E_{n,p}(\tfrac{\pi}{2})}{\Phi_p(\tfrac{\pi}{2})^2} \rp (u) \Biggr| = \oh_{\PP}(1).
\end{multline*}
Evaluating the derivative yields, for any $u \in [v,w]$,
\begin{multline*}
	\psi'_{Q_p} \lp \frac{E_{n,p}(\cdot) \Phi_p(\tfrac{\pi}{2}) - \Phi_p(\cdot) E_{n,p}(\tfrac{\pi}{2})}{\Phi_p(\tfrac{\pi}{2})^2} \rp (u) \\
	= \frac{\Phi_p(Q_p^{-1}(u)) E_{n,p}(\tfrac{\pi}{2}) - E_{n,p}(Q_p^{-1}(u)) \Phi_p(\tfrac{\pi}{2})}{\varphi_p(Q_p^{-1}(u)) \Phi_p(\tfrac{\pi}{2})}.
\end{multline*}
Combining the last three equations concludes the proof.}

%%%%%%%%%%%%%%%%%%%%%%%%%%%%%%%%%%%%%%%%%%%%%%%%%%%%%%%%%%%%%%%%%%%%%%%%
\section{Reduction to $\theta \in [0,\pi/4]$}
\label{app:proofSym}
%%%%%%%%%%%%%%%%%%%%%%%%%%%%%%%%%%%%%%%%%%%%%%%%%%%%%%%%%%%%%%%%%%%%%%%%

We argue that, in the proof of Theorem~\ref{thm:expansionEAM}, we can focus on the region $\theta \in [0,\pi/4]$ by exploiting symmetry.

Let $S : (x,y) \in \R^2 \mapsto (y,x) \in \R^2$ be the symmetry map. By extension, we also define the set-valued map $S(A) := \{S(x,y) : (x,y) \in A\}$ for $A \subseteq \R^2$. Clearly, the measures $S \# P_n$ and $S \# P$ underlying the statement of the theorem satisfy Assumptions~\ref{ass:smoothness} and~\ref{ass:bias} with respect to the measure $\Lambda_S := S \# \Lambda$, which has density $\lambda_S := \lambda \circ S$. The smoothness assumption is trivial. \mdf{For the bias assumption, letting $c_S = c \circ S$ denote the density of $S \# P$, it suffices to note that $S(\mathcal{L}_T) = \mathcal{L}_T$ for any $T \geq 1$ and thus
$$
	\iint_{\mathcal{L}_T} \labs t c_S(tu_1,tu_2) - \lambda_S(u_1,u_2) \rabs \d u_1 \d u_2
	= \mathcal{D}_T(t),
$$
for any $t>0$.
We may thus apply our result for $S \# P_n$ and $S \# P$ for $\theta \in [0,\pi/4]$ once it has been proved for $P_n$ and $P$.}

Now consider $\theta \in [\pi/4,\pi/2]$. Then, up to sets of Hausdorff dimension $1$, \mdf{sets which} are asymptotically negligible for our purposes, we have
\begin{equation}
\label{eq:decomp_set}
	C_{p,\theta} = C_{p,\pi/4} \cup \lp S(C_{p,\pi/4}) \setminus S(C_{p,\pi/2 - \theta}) \rp. 
\end{equation}
Consequently, letting $\hPhi_p^S$ denote the empirical angular measure associated with $S \# P_n$ and $\Phi_p^S$ the angular measure associated with $S \# P$, we have for such $\theta \in [\pi/4,\pi/2]$ that asymptotically, by continuity of the limiting process,
\begin{multline*}
	\sqrt{k} \lp \hPhi_p(\theta) - \Phi_p(\theta) \rp = \sqrt{k} \lp \hPhi_p(\pi/4) - \Phi_p(\pi/4) \rp  + \sqrt{k} \lp \hPhi_p^S(\pi/4) - \Phi_p^S(\pi/4) \rp \\
	- \sqrt{k} \lp \hPhi_p^S(\pi/2 - \theta) - \Phi_p^S(\pi/2 - \theta) \rp + \oh_{\PP}(1),
\end{multline*}
where the $\oh_{\PP}(1)$ term is uniform in $\theta \in [\pi/4,\pi/2]$.
Assuming that the theorem has been shown to hold for $\theta \in [0,\pi/4]$, if $E_{n,p}^S(\theta)$ denotes the asymptotic expansion obtained on the basis of the measures $S \# P_n$ and $S \# P$, we have for $\theta \in [\pi/4,\pi/2]$,
\[
	\sqrt{k} \lp \hPhi_p(\theta) - \Phi_p(\theta) \rp 
	= E_{n,p}(\pi/4) + E_{n,p}^S(\pi/4) - E_{n,p}^S(\pi/2 - \theta) + \oh_{\PP}(1),
\]
where the $\oh_{\PP}(1)$ term is uniform in $\theta \in [\pi/4,\pi/2]$. 
\mdf{It is now sufficient to show that
\begin{equation}
\label{eq:Enpsymmetry}
	E_{n,p}(\pi/4) + E_{n,p}^S(\pi/4) - E_{n,p}^S(\pi/2 - \theta) = E_{n,p}(\theta),
	\qquad \theta \in [\pi/4,\pi/2].
\end{equation}
To do so, we expand $E_{n,p}$ and $E_{n,p}^S$ into three terms as in Theorem~\ref{thm:expansionEAM}.
The expansions of $E_{n,p}(\pi/4)$ and $E_{n,p}(\theta)$ are provided in the theorem. We compute the expansions of $E_{n,p}^S(\pi/4)$ and $E_{n,p}^S(\pi/2-\theta)$.}

For $E_{n,p}^S(\pi/4)$, the first term in the expansion is trivially given by
\[
	\sqrt{k} \lp \tfrac{n}{k} P_n \lp \tfrac{k}{n} S(C_{p,\pi/4}) \rp - \tfrac{n}{k} P\lp \tfrac{k}{n} S(C_{p,\pi/4}) \rp \rp.
\]
\mdf{For the second term in the expansion, the integral over $x < x_p(\pi/4)$ becomes}
%while for the second term in the expansion, we distinguish \mdf{between} $x < x_p(\theta)$ or $x \geq x_p(\theta)$. For the part $x < x_p(\theta)$, 
\[
	\int_0^{2^{1/p}} \lambda_S(x,x) \{w_{2n}(x) - w_{1n}(x)\} \d x = - \int_0^{2^{1/p}} \lambda(x,x) \{w_{1n}(x) - w_{2n}(x) \} \d x,
\]
since $\tan(\pi/4) = \cot(\pi/4) = 1$ and the measure $S_{\#} P_n$ has first marginal tail empirical process $w_{2n}$ and second marginal tail empirical process $w_{1n}$. 
We see that it equals the opposite of \mdf{the integral over $x < x_p(\pi/4)$ in the expansion of $E_{n,p}(\pi/4)$} so that they will cancel out. \mdf{The third term in the expansion is an integral over $x \geq x_p(\pi/4)$, for which we only consider the more complicated case $p<\infty$.} It equals
\begin{multline*}
	\int_{2^{1/p}}^\infty \lambda_S(x,y_p(x)) \lacc w_{2n}(x) y_p'(x) - w_{1n}(y_p(x)) \racc \d x \\
	= \int_1^{2^{1/p}} \lambda(z,y_p(z)) \lacc w_{1n}(z) y_p'(z) - w_{2n}(y_p(z)) \racc \d z,
\end{multline*}
where the second line follows from the change of variable $x = y_p(z)$ which is an involution and satisfies $y_p'(y_p(z)) y_p'(z) = 1$.

For $E_{n,p}^S(\pi/2 - \theta)$, the reasoning is similar. The first term in the expansion is
\[
	\sqrt{k} \lp \tfrac{n}{k} P_n \lp \tfrac{k}{n} S \big( C_{p,\pi/2 - \theta} \big) \rp - \tfrac{n}{k} P\lp \tfrac{k}{n} S \big( C_{p,\pi/2 - \theta} \big) \rp \rp.
\]
\mdf{Since $\tan(\pi/2 - \theta) = \cot(\theta)$, the second term in the expansion is}
\begin{multline*}
	\int_0^{x_p(\pi/2 - \theta)} \lambda_S(x,x\tan(\pi/2 - \theta)) \{w_{2n}(x) \tan(\pi/2 - \theta) - w_{1n}(x \tan(\pi/2 - \theta))\} \d x \\
	 = - \int_0^{x_p(\theta)} \lambda(y,y\tan\theta) \lacc w_{1n}(y) \tan\theta - w_{2n}(y\tan\theta) \racc \d y.
\end{multline*}
The third term is an integral over $x \ge x_p(\pi/2-\theta)$. Again assuming $p < \infty$, using the same change of variable as before and the fact that $y_p(x_p(\pi/2 - \theta)) = x_p(\theta)$, we get
\[
\int_1^{x_p(\theta)} \lambda(z,y_p(z)) \lacc w_{1n}(z) y_p'(z) - w_{2n}(y_p(z)) \racc \d z.
\]

Combining everything, we find \eqref{eq:Enpsymmetry}. Consequently, once the asymptotic expansion stated in the theorem has been shown to hold on the region $\theta \in [0,\pi/4]$, it holds on $\theta \in [0, \pi/2]$.

%%%%%%%%%%%%%%%%%%%%%%%%%%%%%%%%%%%%%%%%%%%%%%%%%%%%%%%%%%%%%%%%%%%%%%%%
\section{Expansion of the second part of the stochastic term}
\label{app:proofi2}
%%%%%%%%%%%%%%%%%%%%%%%%%%%%%%%%%%%%%%%%%%%%%%%%%%%%%%%%%%%%%%%%%%%%%%%%

\revtwo{Let us now consider the case $i=2$ in~\eqref{eq:expansion_stochastic_term_separately} in the proof of Theorem~\ref{thm:expansionEAM} in the paper. Similarly to the case $i=1$, it will be useful to note that for any $\theta \in [0,\pi/4]$,
\[
	\hat{C}_{p,\theta,2} = \lacc (x,y): x \geq x_p(\theta), 0 \leq y \leq \lp x \tan\theta + \frac{z_{n,\theta(x)}}{\sqrt{k}} \rp \wedge \lp y_p(x) + \frac{s_{n,p}(x)}{\sqrt{k}} \rp \racc,
\]
with $z_{n,\theta}$ and $s_{n,p}$ the processes defined in~\eqref{eq:def_z_n} and~\eqref{eq:def_s_n} respectively. For each $m \in \{0,1,2,\ldots, 1/\Delta - 1\}$, we let
\begin{align*}
	S_{m,\Delta,\theta}^+ &:= \sup_{x \in J_\Delta(m)} \frac{s_{n,p}(x) \wedge \lp z_{n,\theta}(x) + \sqrt{k}(x\tan\theta - y_p(x)) \rp }{y_p(x)}, \\
	S_{m,\Delta,\theta}^- &:= \inf_{x \in J_\Delta(m)} \frac{s_{n,p}(x) \wedge \lp z_{n,\theta}(x) + \sqrt{k}(x\tan\theta - y_p(x)) \rp }{y_p(x)},
\end{align*}
\rev{where the intervals $J_\Delta(m)$ were defined in~\eqref{eq:J_intervals} in the paper}. These intervals cover $[2^{1/p}, \infty]$, on which $x\tan\theta \geq x_p(\theta)$ for any $\theta \in [0,\pi/4]$. It follows that on an event whose probability tends to one, $S_{m,\Delta,\theta}^\pm$ is equal to the same expression as in its definition but with the numerator simplified to $s_{n,p}(x)$. In the following, we work with these simplified definitions.
Note that, by~\eqref{eq:w_jn_poids} and~\eqref{eq:quantile_weak_convergences_weighted}, the mean value theorem and the extended continuous mapping theorem~\cite[Theorem~1.11.1]{VVV1996}, the weak convergence 
\begin{equation}
\label{eq:s_n_poids}
	\lacc s_{n,p}(x): x \in [2^{1/p},\infty) \racc 
	\wc
	\lacc W_1(x)y_p'(x) - W_2(y_p(x)): x \in [2^{1/p},\infty) \racc
\end{equation}
holds in $\ell^\infty([2^{1/p},\infty))$. In particular, since $y_p(x) \geq 1$ when $x \geq 2^{1/p}$, we have
\begin{equation}
\label{eq:s_n_boundedness_stochastic}
	\sup_{x \in [2^{1/p}, \infty]} \frac{|s_{n,p}(x)|}{y_p(x)} = \Oh_{\PP}(1).
\end{equation}
Defining
\[
	N_{\Delta,\theta}^\pm := \bigcup_{m=0}^{\Delta - 1} \lacc (x,y) : x \geq x_p(\theta), x \in J_\Delta(m), 0 \leq y \leq y_p(x) \lp 1 + \frac{S_{m,\Delta,\theta}^\pm}{\sqrt{k}} \rp \racc,
\]
we see that for any $\theta \in [0,\pi/4]$ we have $N_{\Delta,\theta}^- \subseteq \hat{C}_{\theta,p,2} \subseteq N_{\Delta,\theta}^+$ and thus
\begin{equation}
\label{eq:bounding_stochastic_2}
	V_{n,\Delta,2}^-(\theta) - R_{n,\Delta,2}(\theta) \leq V_{n,p,2}(\theta) \leq V_{n,\Delta,2}^+(\theta) + R_{n,\Delta,2}(\theta),	
\end{equation}
where
\begin{align}
	V_{n,\Delta,2}^\pm(\theta) &:= \rev{\Wnk(N_{\Delta,\theta}^\pm)}, \label{eq:bounding_processes_2} \\
	R_{n,\Delta,2}(\theta) &:= \sqrt{k} \tfrac{n}{k} P \lp \tfrac{k}{n} (N_{\Delta,\theta}^+ \setminus N_{\Delta,\theta}^-) \rp. \label{eq:rest_2}
\end{align}}

\revtwo{Due to Assumption~\ref{ass:bias} and the behavior of the marginal tail empirical and quantile processes,
\[
	\sup_{\theta \in [0,\pi/4]} \labs R_{n,\Delta,2}(\theta) - \sqrt{k} \Lambda(N_{\Delta,\theta}^+ \setminus N_{\Delta,\theta}^-) \rabs = \oh_{\PP}(1).
\]
Computations in~\cite[Paragraph~A.1]{einmahl2009maximum} show that, with probability tending to one,
\[
	\sup_{\theta \in [0,\pi/4]} \sqrt{k} \Lambda(N_{\Delta,\theta}^+ \setminus N_{\Delta,\theta}^-) 
	\leq 3 \sup_{\theta \in [0,\pi/4]} \max_{m = 0,\ldots,\Delta - 1} (S_{m,\Delta,\theta}^+ - S_{m,\Delta,\theta}^-).
\]
In view of~\eqref{eq:s_n_poids}, we have, for any $\epsilon>0$,
\[
	\lim_{\Delta \ra 0} \limsup_{n \ra \infty} \PP^* \lc \sup_{\theta \in [0,\pi/4]} \max_{m = 0,\ldots,\Delta - 1} (S_{m,\Delta,\theta}^+ - S_{m,\Delta,\theta}^-) > \epsilon \rc = 0,
\]
which implies~\eqref{eq:banane} with $R_{n,\Delta,1}$ replaced by $R_{n,\Delta,2}$.}

\revtwo{By a similar argument as for $i=1$, we have
\[
	\sup_{\theta \in [0, \pi/4]} \rho(N_{\Delta,\theta}^\pm, C_{p,\theta,2}) \leq \frac{3}{\sqrt{k}} \sup_{\theta \in [0, \pi/4]} \max_{m = 0,\ldots,\Delta - 1} |S_{m,\Delta,\theta}^\pm| = \oh_{\PP}(1),
\]
and thus~\eqref{eq:expansion_stochastic_term_separately} holds for $i=2$.}

\end{document}